\UseRawInputEncoding
\documentclass[11pt,a4paper]{article}
\usepackage{graphicx}
\usepackage{indentfirst}
\usepackage{float}
\usepackage{mathrsfs}
\usepackage{fancyhdr}
\usepackage{amsfonts,amssymb,amsmath,amsthm}
\usepackage{hyperref}
\newtheorem{thm}{Theorem}[section]
\newtheorem{lemma}[thm]{Lemma}
\newtheorem{cor}[thm]{Corollary}
\newtheorem{prop}[thm]{Proposition}

\newtheoremstyle{rem}{10pt}{10pt}{\rmfamily}{}{\bfseries}{.}{.5em}{}
\theoremstyle{rem}
\newtheorem{rem}[thm]{Remark}
\newtheorem{defi}[thm]{Definition}

\title{\bf Global Cauchy problems  for the nonlocal (derivative) NLS in $E^s_\sigma$}
\author{Jie Chen, Yufeng Lu and Baoxiang Wang\footnote{Corresponding author}}
\date{}

\begin{document}

\maketitle
\begin{abstract}
		 We consider the Cauchy problem for the (derivative) nonlocal NLS in super-critical function spaces $E^s_\sigma$  for which the norms are defined by
$$
\|f\|_{E^s_\sigma} = \|\langle\xi\rangle^\sigma 2^{s|\xi|}\widehat{f}(\xi)\|_{L^2}, \ s<0, \ \sigma \in \mathbb{R}.
$$
Any Sobolev space  $H^{r}$  is a subspace of $E^s_\sigma$, i.e., $H^r \subset E^s_\sigma$ for any $ r,\sigma \in \mathbb{R}$ and $s<0$.  Let $s<0$ and $\sigma>-1/2$ ($\sigma >0$) for the nonlocal NLS (for the nonlocal derivative NLS). We show the global existence and uniqueness of the solutions if the initial data belong  to $E^s_\sigma$ and their Fourier transforms are supported in $(0, \infty)$,  the smallness conditions on the initial data in $E^s_\sigma$ are not required for the global solutions.\\

		
	\end{abstract}

\section{Introduction}
	We  consider the Cauchy problem for the nonlocal nonlinear Schr\"odinger equation (NNLS)
\begin{equation}\label{NNLS}
		 \mathrm{i} u_t + \partial^2_x u + \alpha\, u^2u^* =0,  \  u(x,0) = u_0(x),
\end{equation}
and its derivative version (NdNLS)
\begin{equation}\label{NdNLS}
		 \mathrm{i} u_t + \partial^2_x u + \alpha\,  u u^*\partial_x u  =0,  \  u(x,0) = u_0(x), 
\end{equation}
where $u(x,t)$ is a complex valued function of $(x,t)\in \mathbb{R}\times[0,T)$ for some $0<T\leq \infty$, $u^*(x,t) = \bar{u}(-x, t )$ and $\bar{u}$ stands for the  conjugate complex number of $u$, $\alpha \in \mathbb{R}$.
A  more general nonlocal  derivative NLS (gNdNLS) is the following
\begin{equation}\label{gNdNLS}
		 \mathrm{i} u_t + \partial^2_x u + \alpha\,  u u^*\partial_x u + \beta\,  u^2 \partial_x u^*   =0,  \  u(x,0) = u_0(x).
\end{equation} 
NNLS was introduced by Ablowitz and Musslimani \cite{AbMu2013,AbMu2016} in 2013  when they studied the exactly solvable models of nonlinear wave propagation. It is known that the solution of NNLS satisfies the following conservation laws:
\begin{align}
M(u) & =\int_{\mathbb{R}} u(x,t)\bar{u}(-x, t ) dx = M(u_0),  \label{masscon}\\
E(u) &=  \int_{\mathbb{R}} \left( u_x (x,t)\bar{u}_x (-x, t ) + \frac{\alpha}{2} u^2(x,t)\bar{u}^2(-x, t ) \right) dx = E(u_0)  \label{energycon}
\end{align}
As an integrable Hamiltonian system, the inverse scattering transform and  soliton solutions of NNLS were discussed in \cite{AbMu2013,AbMu2016,FeLuAbMu2018,GuPe2018}. Gadzhimuradov and Agalarov  \cite{GaAg2016} showed that NNLS is gauge equivalent to the unconventional system of coupled Landau-Lifshitz
equations. The long-time asymptotic behavior for NNLS with a family of step-like, and weighted Sobolev initial data were studied in Rybalko and Shepelsky \cite{RySh2019,RySh2021,RySh2021JDE} and in Li, Yang and Fan \cite{LiYaFa2021}, respectively.

NdNLS was proposed in Zhou \cite{Zh2018}, Shi, Shen and Zhao \cite{ShShZh2019} as a nonlocal version of the derivative nonlinear Schr\"odinger equation which is integrable, invariant under spacial reversion together with complex conjugation and has the conservation
\begin{align}
M(u) & =\int_{\mathbb{R}} u(x,t)\bar{u}(-x, t ) dx = M(u_0). \label{massdnls}
\end{align}
In \cite{Zh2018}, by constructing its Darboux transformations, the explicit expressions of solutions are derived
from zero seed solutions.  Using the nonlocal gauge transformation, Shi, Shen and Zhao \cite{ShShZh2019} considered all of the possible nonlocal versions of the derivative nonlinear Schr\"odinger equations.

In the radial case $u(x,t)=u(-x,t)$, NNLS reduces to the following cubic NLS:
\begin{equation}\label{rNNLS}
		 \mathrm{i} u_t + \partial^2_x u + \alpha\, |u|^2u =0,  \  u(x,0) = u_0(x).
\end{equation}
The cubic NLS has been sufficiently studied in recent years.  Tsutsumi \cite{Tsu1987} obtained the global well-posedness of \eqref{rNNLS} in $L^2$, Harrop-Griffiths,  Killip and Visan \cite{HaKiVi2020} have shown the sharp global well-posedness of \eqref{rNNLS} in $H^\sigma$ for any $\sigma >-1/2$, where the sharpness means that \eqref{rNNLS} is ill-posed in $H^\sigma$ if $\sigma \leq -1/2$. The results of \cite{Tsu1987,HaKiVi2020} in the radial case also adapt to the NNLS \eqref{NNLS}.  However,  for the non-radial initial data, it seems difficult to develop the results of \cite{Tsu1987,HaKiVi2020}  to the NNLS \eqref{NNLS}, since the conservation laws cannot provide the {\it a priori} upper bounds of the solutions in $L^2$, which was applied in \cite{Tsu1987} to show the global well-posedness in $L^2$, and the global solutions as in \cite{HaKiVi2020} with initial data in $H^\sigma$ with $\sigma >-1/2$ seems difficult to be obtained for the NNLS in the non-radial case. Eq. \eqref{NdNLS} is related to the following derivative NLS
\begin{equation}\label{rNdNLS}
		 \mathrm{i} u_t + \partial^2_x u \pm  \mathrm{i} \, \partial_x (|u|^2u)   =0,  \  u(x,0) = u_0(x).
\end{equation}
The global well-posedness of Eq.~\eqref{rNdNLS} has been extensively studied in lower regularity spaces $H^\sigma$, $0<\sigma \leq 1/2$ (see \cite{BaPe20,GuWu2017,HaKiVi2021} and references therein). Guo and Wu consider the global well-posedness of Eq.~\eqref{rNdNLS} in $H^{1/2}$ with the condition $\|u_0\|_2 <\sqrt{4\pi}$.   Bahouri and Perelman \cite{BaPe20} removed the smallness condition on $\|u_0\|_2 <\sqrt{4\pi}$ in \cite{GuWu2017}.  Harrop-Griffiths, Killip and Visan \cite{HaKiVi2021} showed the global well-posedness for the initial data in $H^\sigma$, $1/6<\sigma<1/2$. In \cite{BaPe20,HaKiVi2021}, the integrable structure of \eqref{rNdNLS} seems to be important for the global well-posedness without smallness condition on initial data. However, the global solutions of NdNLS cannot be obtained by following the same way as the derivative NLS, since the conservation law or the integrable structures cannot formulate useful norm structure for the NdNLS.

In this paper, we will use a different way to study the global existence and uniqueness of solutions for NNLS and NdNLS by introducing a class of very rough function spaces $E^s_\sigma$ as the existence spaces of solutions. Now let us define the function spaces $E^s_\sigma$.

\subsection{Function spaces $E^s_\sigma$}

Let  $ \mathscr{S}$ be the Schwartz space  and $ \mathscr{S}'$ be its dual space. We write
$$
p_\lambda(f) = \sup_{x\in \mathbb{R}}e^{\lambda |x|}|f(x)|, \quad q_{\lambda}(f)= \sup_{\xi\in \mathbb{R}} e^{\lambda|\xi|}|\widehat{f}(\xi)|,
$$
\begin{equation*}
		\mathscr{S}_1:=\{f\in \mathscr{S}: p_\lambda(f)+q_\lambda(f)<\infty, ~\forall~\lambda>0 \}.
\end{equation*}	
	$ \mathscr{S}_1$ equipped with the system of semi-norms $\{p_\lambda+q_\lambda\}_{\lambda>0}$ is a complete locally convex linear topological space, which is said to be the Gelfand-Shilov space, cf. \cite{GeSh1968}. We denote by $ \mathscr{S}_1'$ the dual space of $ \mathscr{S}_1$.  One easily sees that
$$ \mathscr{S}_1 \subset  \mathscr{S}, \ \  \mathscr{S}' \subset  \mathscr{S}'_1.$$
$\mathscr{S}_1$ contains the translations and modulations of Gaussian $e^{-\mathrm{i}m x} e^{- |x-n|^2/2}$ and their linear combinations, which is dense in all Sobolev spaces $H^\sigma$ (cf. \cite{GeSh1968}).
The Fourier transforms on $ \mathscr{S}_1'$ can be defined by duality (cf. \cite{FeGrLiWa2021}), namely, for any $f\in  \mathscr{S}'_1$, its Fourier transform $\mathscr{F} f = \widehat{f}$ satisfies
$$
\langle \mathscr{F} f, \,  \varphi \rangle = \langle f, \,  \mathscr{F} \varphi \rangle, \ \ \forall \ \varphi  \in \mathscr{S}_1.
$$
Now we can introduce the function spaces $E^s_\sigma$ defined in \cite{ChWaWa2022}. For $s, \sigma \in \mathbb{R}$, we denote
		\begin{equation*}
			E^s_\sigma :=\{f\in \mathscr{S}_1': 	\langle\xi \rangle^{\sigma}  2^{s|\xi|}\widehat{f}(\xi)\in L^2(\mathbb{R})\}
		\end{equation*}
		 for which the norm is defined by
$\|f\|_{E^{s}_\sigma} = 	\|\langle\xi \rangle^{\sigma} 2^{s|\xi|}\widehat{f}(\xi)\|_{2}.$ $E^s_\sigma$ is a natural generalization of Sobolev spaces $H^\sigma:=(I-\Delta)^{-\sigma/2} L^2(\mathbb{R}^d)$ for which the norm is defined by
\begin{align} \label{spaceHs}
  \|f\|_{H^\sigma} := \|\langle \xi\rangle^\sigma \widehat{f} \|_{2}.
\end{align}
One easily sees that $H^\sigma = E^0_\sigma$. $E^s_\sigma$ is much rougher than $H^\sigma$  in the case $s<0$ and let us observe the following inclusion (cf. \cite{ChWaWa2022}).

\begin{prop} \label{embedding}
Let $s<0$ and $r,\sigma \in \mathbb{R}$. Then we have $H^{r} \subset E^s_\sigma$.
\end{prop}

By Proposition \ref{embedding}, we see that $\cup_{r\in \mathbb{R}} H^r$ is a subset of $E^s_\sigma$ if $s<0$.
 The aim of this paper is to study the Cauchy problem for NNLS and NdNLS with initial data in $E^s_\sigma$ with $s<0$.

\subsection{Critical and super-critical spaces}

 For a nonlinear evolution equation, a Sobolev space $H^{\sigma_c}$ is said to be its critical space if it is locally or globally well-posed in $H^\sigma$ for $\sigma>\sigma_c$ and ill-posed in $H^\sigma$ for $\sigma<\sigma_c$.   $H^\sigma$ with $\sigma>\sigma_c$ ($\sigma<\sigma_c$) is said to be the subcritical (supercritical) Sobolev space. Any Banach function space $X$ satisfying $H^\sigma\subset X$ for some $\sigma<\sigma_c$   is said to be a {\it supercritical space}.

From the scaling arguments, we see that NNLS has a scaling critical Sobolev space $ H^{-1/2}$, which means that if $u$  solves NNLS, so does $u_\lambda = \lambda u(\lambda x, \lambda^2 t)$  and $u_\lambda|_{t=0}$ have the equivalent norms in $  H^{-1/2}$ for all $\lambda>1$. Similarly, $L^2$ is the scaling critical space for the NdNLS.  From this point of view, $H^{-1/2}$ and $L^2$  are said to be the scaling critical spaces of NNLS and NdNLS, respectively. Following the same way as the NLS and dNLS, we see that NNLS (NdNLS) is ill-posed in $H^\sigma$ for $\sigma <-1/2$ ($\sigma<0$) (see Oh \cite{Oh2017} for instance). It follows that $E^s_\sigma$ with $s<0$ is also a supercritical space for NNLS and NdNLS.

\subsection{Main result}

There are some recent works which has been devoted to the study of the global well posedness for a class of nonlinear evolution equations in supercritical spaces (cf. \cite{Li2021,FeGrLiWa2021,ChWaWa2022}). In this paper, we are interested in the NNLS and NdNLS in the super-critical spaces $E^s_\sigma$ and we have the following results.
	\begin{thm}\label{mainresult}
	Let $s \leq 0, \ -1/2< \sigma \leq 0$, $u_0 \in E^s_\sigma$ with $\mathrm{supp}~\widehat{u}_0 \subset [ 0,\infty) $. Then  there exists a $T>0$ such that NNLS \eqref{NNLS} has a unique solution $u\in L^\infty(0,T; E^s_\sigma) \cap X^s_{\sigma,\Delta, T}$, where $X^s_{\sigma,\Delta, T} $ is defined in Definition \ref{workingspace}.
	Moreover, if $s<0$, $\mathrm{supp}~\widehat{u}_0 \subset [\varepsilon_0,\infty) $ for some $\varepsilon_0>0$,  then there exists some $j_0 := j_0(\alpha, s, \varepsilon_0, \sigma, \|u_0\|_{E^s_\sigma} )  \in \mathbb{N}$ such that NNLS has a unique global solution
$$
u \in  L^\infty (0, 2^{ \sqrt{j}}; E^{sj}_\sigma ) \cap X^{sj}_{\sigma, \Delta, 2^{\sqrt{j}} }, \ \  \forall \ j\geq j_0
$$
satisfying
$$
  \sup_{j \geq j_0 } \|u \|_{L^\infty(0, 2^{ \sqrt{j}}; E^{sj}_\sigma) \cap X^{sj}_{\sigma, \Delta, 2^{\sqrt{j}} }} <\infty.
$$
\end{thm}

\begin{rem}
Theorem \ref{mainresult} needs several remarks.
\begin{itemize}


\item[(i)]  Theorem \ref{mainresult} contains the following rather rough initial data.
For any $k\in \mathbb{Z}_+$, $A\in \mathbb{C}$,  let $\delta(x)$ be the Dirac measure and
   \begin{align}
    & u_0(x)= A\, e^{\mathrm{i}x} \frac{d^k}{dx^k} F(x), \ \  or \label{examples1} \\
   & u_0(x)=  e^{\mathrm{i}x} \sum^\infty_{m=0} \frac{(-{\rm i}\lambda)^m}{m!} \frac{d^m}{dx^m} F(x), \ \ |\lambda|< |s|, \label{examples2}
  \end{align}
where $F(x) = \lim_{\varepsilon\to 0} \frac{1}{x+2\pi \mathrm{i} \varepsilon}  = \left(\delta (x) + p.v. \frac{2 \, \mathrm{ i}}{x}\right)$ denotes  the Sokhotski-Plemelj distribution, which plays a crucial role in the Sokhotski-Plemelj formula (cf. \cite{SaNa2014}, Page 374).

\item[(ii)] Since $(L^1(0,T; E^{-s}_{-\sigma}))^*= L^\infty(0,T; E^{s}_{\sigma})$ and $C^\infty_0 (0,\infty; \mathscr{S}_1) \subset L^1(0,\infty; E^{-s}_{-\sigma})$, one sees that
   \begin{equation}\label{NNLSinn}
		 \langle \mathrm{i} u_t + \partial^2_x u + \alpha\, u^2u^* , \psi \rangle =0, \ \ \forall \psi \in C^\infty_0 (0,\infty; \mathscr{S}_1),
\end{equation}
which means that the solution $u$ obtained in Theorem \ref{mainresult} satisfies NNLS in the distribution space $\mathscr{D}'(0,\infty; \mathscr{S}'_1)$.

\item[(iii)] Condition ${\rm supp} \ \widehat{u}_0 \subset [0, \infty)$ is necessary for the result of Theorem \ref{mainresult}, we can give an arbitrarily small initial value $u_0\in E^s_\sigma$ ($s<0, \sigma>0$) whose Fourier transform is supported in $\mathbb{R}$  such that the second iteration of the solution is quite large in $E^{s'}_{\sigma'}$ (for any $s', \sigma' \in \mathbb{R}$) at any small time,    see Section \ref{Illposed}.

\end{itemize}
\end{rem}

Using a nonlocal gauge transform
\begin{align} \label{iimaggauge1}
v(t,x) = u(t,x)\exp\left(- \mathrm{i}  \delta  \partial^{-1}_x  (u  {u}^* )   \right),
\end{align}
 Shi, Shen and Zhao \cite{ShShZh2019} considered various equivalent versions for the nonlocal derivative NLS. For our purpose we introduce the following nonlocal gauge transform
\begin{align} \label{rgauge1}
v(t,x) =  u(t,x)\exp\left( \frac{\alpha}{2} \partial^{-1}_x  (u  {u}^* ) \right), \ \ \partial^{-1}_x = \frac{1}{2} \left(\int^x_{-\infty}- \int^\infty_x \right).
\end{align}
Using the nonlocal gauge transform \eqref{rgauge1}, we can show that NdNLS \eqref{NdNLS} is equivalent to the following (see Section \ref{GloNDNLS})
\begin{equation}\label{aNdNLSequiv}
		 \mathrm{i} v_t + \partial^2_x v - \alpha\,  v^2 \partial_x v^* -\frac{\alpha^2}{2} v^3 (v^*)^2=0,  \  v(x,0) = v_0(x).
\end{equation}

\begin{thm}\label{mainresult2}
	Let $s \leq 0, \   \sigma > 0$, $v_0 \in E^s_\sigma$ with $\mathrm{supp}~\widehat{v}_0 \subset [ 0,\infty) $. Then  there exists a $T>0$ such that NdNLS \eqref{aNdNLSequiv} has a unique solution $v\in L^\infty(0,T; E^s_\sigma) \cap \mathcal{X}^s_{\sigma,\Delta, T}$, where $\mathcal{X}^s_{\sigma,\Delta, T}$ is defined in \eqref{Xssigmadelta}.
	Moreover, if $s<0$, $\mathrm{supp}~\widehat{v}_0 \subset [\varepsilon_0,\infty) $ for some $\varepsilon_0>0$,  then there exists some $j_0 := j_0(\alpha, s, \varepsilon_0, \sigma, \|v_0\|_{E^s_\sigma} )  \in \mathbb{N}$ such that NdNLS \eqref{aNdNLSequiv} has a unique global solution
\begin{align}
& v \in  L^\infty(0, 2^{ \sqrt{j}}; E^{sj}_\sigma ) \cap X^{sj}_{\sigma, \Delta, 2^{\sqrt{j}} }, \ \  \forall \ j\geq j_0,
 \label{1ndnlsglobal}\\
&  \sup_{j \geq j_0 } \|v\|_{L^\infty(0, 2^{ \sqrt{j}}; E^{sj}_\sigma) \cap X^{sj}_{\sigma, \Delta, 2^{\sqrt{j}} }} <\infty.
\label{2ndnlsglobal}
\end{align}
\end{thm}

On the basis of Theorem \ref{mainresult2}, one has that

	\begin{cor}\label{mainresult3}
	Let $s < 0, \   \sigma > 0$, $u_0 \in E^s_\sigma$ with $\mathrm{supp}~\widehat{u}_0 \subset [\varepsilon_0,\infty) $ for some $\varepsilon_0>0$. Let
$
v_0 = u_0 \exp\left(  \frac{\alpha}{2} \partial^{-1}_x  (u_0  {u}_0^* ) \right).
$
Then NdNLS \eqref{aNdNLSequiv} has a unique global solution $v$ satisfying \eqref{1ndnlsglobal} and  \eqref{2ndnlsglobal}. Moreover,
$$
u(t,x) = v(t,x) \exp\left( - \frac{\alpha}{2} \partial^{-1}_x  (v  {v}^* ) \right)
$$
is the global solution of NdNLS \eqref{NdNLS} and there exists some $j_0 := j_0(s, \varepsilon_0, \sigma, \|u_0\|_{E^s_\sigma} )  \in \mathbb{N}$ such that
\begin{align*}
  u \in  L^\infty(0, 2^{ \sqrt{j}}; E^{sj}_\sigma )\cap \mathcal{X}^{sj}_{\sigma, \Delta, 2^{ \sqrt{j}}} , \ \   \sup_{j \geq j_0 } \|u\|_{L^\infty(0, 2^{ \sqrt{j}}; E^{sj}_\sigma) \cap \mathcal{X}^{sj}_{\sigma, \Delta, 2^{ \sqrt{j}}} } <\infty, \ \  \   \forall \ j\geq j_0.
\end{align*}
\end{cor}
 Similar to NNLS, condition ${\rm supp} \ \widehat{u}_0 \subset [0, \infty)$ is also necessary for the result of Theorem \ref{mainresult2} and Corollary \ref{mainresult3}. There is no difficulty to generalize the results of Theorem \ref{mainresult2} and Corollary \ref{mainresult3} to gNdNLS \eqref{gNdNLS}, see Section \ref{sectgNdNLS}.

\subsection{Main ideas}

We now indicate the crucial ideas to solve the NNLS in $E^s_\sigma$.  One of the main difficulties to solve nonlinear dispersive equation in supercritical Sobolev spaces $H^\sigma$ lies in the fact that the ill-posedness occurs in supercritical Sobolev spaces $H^\sigma$. However, $E^s_\sigma$ type spaces have some good algebraic structures when the frequency is localized in the half-line, so that the nonlinear estimates become available in $E^s_\sigma$ type spaces. So, one can get a local well-posedness result of NNLS for initial data in $E^s_\sigma$ if their Fourier transforms are supported in the half-line.

Observing the scaling solution $u_\lambda (t,x)= \lambda u(\lambda^2 t, \lambda x)$ of NNLS in super-critical space $ H^\sigma$ with $\sigma<-1/2$, we have
\begin{align*}
\|u_\lambda|_{t=0}\|_{ H^\sigma}  \lesssim   \lambda^{\sigma +1/2} \|u_0\|_{H^\sigma} \to 0, \ \ \lambda \to \infty.
\end{align*}
It follows that the scaling solution can have very small initial data in supercritical Sobolev spaces. The above observation is also adapted to the supercritical space $E^s_\sigma$ ($s<0$), $u_\lambda|_{t=0}$ will vanish in $E^s_\sigma$  if ${\rm supp} \,\widehat{u}_0 \subset (0,\infty)$ and $\lambda \to \infty$. Moreover, the small initial data  $u_\lambda|_{t=0}$ will allow the existence time of the solutions becomes larger and larger as $\lambda \nearrow \infty$.  Taking $\lambda\to \infty$, we can get the solution exists at any time, see Section \ref{sectNNLS} for details.

\subsection{Notations and organization}

 Throughout this paper, we denote by $L^p_x$ the Lebesgue space on $x\in \mathbb{R}$ and write
$ \|f\|_p :=\|f\|_{L^p_x}. $
For any function $g$ of $(t,x) \in \mathbb{R}_+ \times \mathbb{R}$, we denote
$$
\|f\|_{L^\gamma_t L^p_x} = \|\|g\|_{L^p_x}\|_{L^\gamma_t},
$$
where $L^\gamma_t$ can be defined in a similar way as $L^p_x$ by replacing $\mathbb{R}^d$ with $\mathbb{R}_+ $. Let us write $\langle \nabla\rangle^s = \mathscr{F}^{-1} \langle \xi\rangle^s \mathscr{F}$, $2^{s|\nabla|}  = \mathscr{F}^{-1} 2^{s|\xi|} \mathscr{F}$.
We write $|x|=|x_1|+...+ |x_d| $, $|x|_\infty = \max_{1\leq i\leq d} |x_i|$ and $\langle x\rangle=(1+x^2_1+...+x^2_d)^{1/2}$ for $x=(x_1,...,x_d)$.  We will use the following notations. $C\ge 1, \ c\le 1$ will denote constants which can be different at different places, we will use $A\lesssim B$ to denote   $A\leqslant CB$; $A\sim B$ means that $A\lesssim B$ and $B\lesssim A$.   We denote by  $\mathscr{F}^{-1}f$ the inverse Fourier transform of $f$.  For any $1\leq p \leq \infty$,   $l^p$  stands for the (sequence) Lebesgue space.

The paper is organized as follows. In Section \ref{sectNNLS} we  show a global existence and uniqueness result for NNLS. In Section \ref{GloNDNLS}, by introducing a nonlocal gauge transform, we get an equivalent NdNLS and a global existence and uniqueness result is proven. Finally, in Section \ref{sectgNdNLS} we consider a general NdNLS and point out its global existence and uniqueness of solutions can be obtained by following the same way as in Section \ref{GloNDNLS}.

	\section{Nonlocal NLS} \label{sectNNLS}

\subsection{$X^s_{\sigma,\Delta}$ and  $Y^s_{\sigma,\Delta}$}
	We need the $U^p$ and $V^p$ spaces which were introduced in \cite{KoTa2005} (see also \cite{HaHeKo2009}).
		Let $\mathcal{Z}$ be the set of finite partitions $-\infty=t_0<t_1<\cdots<t_K = \infty$. $1\leq p<\infty$. For $\{t_k\}_{k=0}^K\subset \mathcal{Z}$ and $\{\phi_k\}_{k=0}^{K-1}\subset L^2$ with $\sum_{k=0}^{K-1}\|\phi_k\|_{L^2}^p = 1$ and $\phi_0 = 0$, we call the function $a:\mathbb{R}\rightarrow L^2$ given by $a = \sum_{k=1}^K \chi_{[t_{k-1},t_k)}\phi_{k-1}$ a $U^p$-atom. Define
		\begin{equation*}
			U^p:=\left\{u = \sum_{j=1}^\infty \lambda_ja_j: a_j~\mbox{is a}~U^p\mbox{-atom}, \lambda_j\in \mathbb{C}~\mbox{such that}~\sum_{j=1}^\infty |\lambda_j|<\infty\right\}
		\end{equation*}
		with the norm
		\begin{equation}
			\|u\|_{U^p}:=\inf\left\{\sum_{j=1}^\infty|\lambda_j|: u=\sum_{j=1}^\infty \lambda_j a_j, \lambda_j\in \mathbb{C}, a_j ~\mbox{is a}~U^p\mbox{-atom}\right\}.
		\end{equation}
Let $1\leq p<\infty$, the space $V^p$ is defined as the normed space of all functions $v:\mathbb{R}\rightarrow L^2$ such that $v(-\infty):=\lim_{t\rightarrow -\infty} v(t)$ exists and for which the norm
\begin{equation*}
\|v\|_{V^p}:= \sup_{\{t_k\}_{k=0}^K\in \mathcal{Z}} \left(\sum_{k=1}^K\|v(t_k)-v(t_{k-1})\|_{L^2}^p\right)^\frac{1}{p}
\end{equation*}
is finite, where we use the convention $v(\infty) = 0$. Let $V^p_{-, rc}$ denote all $v\in V^p$ which are right-continuous and $v(-\infty) = 0$.
We define
$  \|u\|_{U^p_{\Delta}} = \|e^{-{\rm i} t \Delta} u \|_{U^p}, \
  \|u\|_{V^p_{\Delta}} = \|e^{-{\rm i} t \Delta} u \|_{V^p}. $
Recall that Besov type Bourgain's spaces $\dot X^{s, b, q}$ are defined by
$$
\|u\|_{\dot X^{s,b,q}} := \left\| \|\chi_{|\tau+\xi^2|\in [2^{j-1}, 2^j)} |\xi|^{s} |\tau+\xi^2|^{b} \widehat{u}(\tau,\xi) \|_{L^2_{\xi,\tau}}  \right\|_{\ell^q_{j\in \mathbb{Z}}}.
$$

	\begin{prop} {\rm (\cite{HaHeKo2009})}\label{basicproperty}
		Let $1\leq p<q<\infty$.
		\begin{itemize}
			\item[$\mathrm{(i)}$] $U^p, V^p$ are Banach spaces. $V^p_{-,rc}$ is a closed subspace of $V^p$.
			\item[$\mathrm{(ii)}$] The embedding $U^p\subset V_{-,rc}^p \subset L^\infty(\mathbb{R}, L^2)$ is continuous.
			\item[$\mathrm{(iii)}$] The embedding $V_{-,rc}^p\subset U^q$ is continuous.
\item[$\mathrm{(iv)}$] $\dot X^{0, 1/2, 1} \subset U^2_{\Delta} \subset V^2_{rc,-,\Delta} \subset \dot X^{0, 1/2, \infty}$.
		\end{itemize}
	\end{prop}
By the last inclusion of (iv) in Proposition \ref{basicproperty}, we see that

\begin{lemma}[\rm Dispersion Modulation Decay]
Suppose that the dispersion modulation $|\tau+\xi^2| \gtrsim \mu$ for a function $u\in L^2_{x,t}$, then we
\begin{align}
  \|u \|_{L^2_{x,t} }  \lesssim \mu^{-1/2} \|u\|_{V^2_\Delta}. \label{dispersiondecay}
\end{align}
\end{lemma}

\begin{prop} {\rm (\cite{HaHeKo2009})} \label{Updual}
		{\rm (Duality)} Let $1\leq p   <\infty$, $1/p+1/p'=1$.  Then $(U^p)^* = V^{p'}$ in the sense that
\begin{align}
T: V^{p'}  \to (U^p)^*; \ \ T(v)=B(\cdot,v), \label{dual}
\end{align}
is an isometric mapping.  The bilinear form $B: U^p\times V^{p'}$ is defined in the following way: For a partition $\mathrm{t}:= \{t_k\}^K_{k=0} \in \mathcal{Z}$, we define
 \begin{align}
B_{\mathrm{t}} (u,v) = \sum^K_{k=1}( u(t_{k-1}), \ v(t_k)-v(t_{k-1})). \label{dual2}
\end{align}
Here $( \cdot, \cdot )$ denotes the inner product on $L^2$. For  any $u\in U^p$, $v\in V^{p'}$, there exists a unique number $B(u,v)$ satisfying the following property. For any $\varepsilon>0$, there exists a partition $\mathrm{t}$ such that
$$
|B(u,v)- B_{\mathrm{t}'} (u,v)| <\varepsilon, \ \ \forall \  \mathrm{t}'    \supset \mathrm{t}.
$$
Moreover,
$$
|B(u,v)| \leq \|u\|_{U^p} \|v\|_{V^{p'}}.
$$
In particular, let $u\in V^1_{-}$ be absolutely continuous on compact interval, then for any $v\in V^{p'}$,
$$
 B(u,v) =\int ( u'(t), v(t)) dt.
$$
\end{prop}
\begin{prop}[Strichartz estimate] \label{strichartz}
		Let $(q,r)$ be a admissible pair, i.e.,
		$$\frac{2}{q}+\frac{1}{r} = \frac{1}{2}, \quad 4\leq q\leq \infty.$$
		Then, we have
		$$\|e^{it\Delta}u_0\|_{L^q_tL^r_x}\lesssim \|u_0\|_{L^2}.$$
		For any $u\in U^q_\Delta, q<\infty$, we have
		\begin{equation}
			\|u\|_{L^q_tL^r_x}\lesssim \|u\|_{U^q_\Delta}.
		\end{equation}
\end{prop}

\begin{defi}\label{workingspace}
For $s\in \mathbb{R}$, define
$$
X^s_{\sigma,\Delta} :=\{u\in \mathscr{S}' (\mathbb{R},\mathscr{S}'_1):   \|\langle\nabla\rangle^\sigma 2^{s|\nabla|} u\|_{U^2_{\Delta}} <\infty\},
$$
$$
Y^s_{\sigma,\Delta} :=\{u\in \mathscr{S}' (\mathbb{R},\mathscr{S}'_1):   \|\langle\nabla\rangle^\sigma 2^{s|\nabla|}   u\|_{V^2_{\Delta}} <\infty\},
$$
 where $\|u\|_{U^2_\Delta} = \|e^{-it\Delta}u\|_{U^2}$, $\|v\|_{V^2_\Delta} = \|e^{-it\Delta}v\|_{V^2}$. $X^s_{\sigma,\Delta}$ and $ Y^s_{\sigma,\Delta}$ are equipped with norms
\begin{align*}
&		\|u\|_{X^s_{\sigma,\Delta}}: =  \|\langle\nabla\rangle^\sigma 2^{s|\nabla|} u\|_{U^2_{\Delta}} , \ \  \ \
 \|v\|_{Y^{ s}_{\sigma,\Delta} }:= \|\langle\nabla\rangle^\sigma 2^{s|\nabla|}   u\|_{V^2_{\Delta}},
\end{align*}
respectively. For any $T>0$, define
$$
X^s_{\sigma,\Delta,T}:=\{u\in X^s_{\sigma,\Delta} :\ u(t) = e^{i(t-T)\Delta}u(T), \ \forall~ t\geq T\}.
$$
\end{defi}

 For the free solution of the Schr\"odinger equation $u(x,t)= e^{{\rm i}t\Delta} u_0$, $\widehat{u}(\xi,\tau)$ is supported on a curve $\tau+\xi^2=0$, which is said to be the {\it dispersion relation}. For the solution $u$ of NNLS, $\widehat{u}(\xi,\tau)$ can be supported in $(\xi,\tau) \in \mathbb{R}^2$,  we need to consider the size of $|\xi^2+\tau|$, which is said to be the {\it dispersion modulation}. By the second inclusion in (ii) of Proposition \ref{basicproperty} we have
\begin{align}
X^s_{\sigma,\Delta} \subset L^\infty (\mathbb{R}, E^s_{\sigma}).   \label{mod-space2aaaa}
\end{align}

	\begin{rem}
		By Proposition \ref{Updual}, we have $(X^s_{\sigma,\Delta})^* = Y^{-s}_{-\sigma,\Delta}$.
Taking
$$
u(t) = \int_0^t e^{i(t-\tau)\Delta}F(\tau)d\tau,
$$
we have
		\begin{equation}\label{dualuse}
			\|u\|_{X^s_{\sigma,\Delta}} = \sup_{\|v\|_{Y^{-s}_{-\sigma,\Delta}}\leq 1 }
		\int_{\mathbb{R}} (F(t), v(t))dt. 
		\end{equation}
	To make sure that the above formula is meaningful and show \eqref{dualuse}, we assume that $F\in L^\infty(\mathbb{R}, L^2)$, $\widehat{F(t)}$ has compact support set and equals to $0$ except $t$ belong to some compact interval, then
	\begin{align*}
		\|u\|_{X^s_{\sigma,\Delta}} &= \left\|\langle\nabla\rangle^{\sigma} 2^{s|\nabla|} e^{-it\Delta} u \right\|_{U^2} \\
		 &= \sup_{\|v \|_{V^2 }\leq 1 } B \left(\langle\nabla\rangle^{\sigma} 2^{s|\nabla|} e^{-it\Delta} u, \ v \right)\\
		 &= \sup_{\|v \|_{V^2 }\leq 1 } B \left ( e^{-it\Delta} u, \ \langle\nabla\rangle^{\sigma} 2^{s|\nabla|} v \right)\\
          & = \sup_{\|v \|_{V^2}\leq 1 } \int_{\mathbb{R}} \left(e^{-it\Delta} F(t), \ \langle\nabla\rangle^{\sigma} 2^{s|\nabla|} v \right)dt\\
		  & = \sup_{\|v \|_{V^2}\leq 1 } \int_{\mathbb{R}}\left (F(t), \ \langle\nabla\rangle^{\sigma} 2^{s|\nabla|} e^{it\Delta} v \right)dt\\
		 & \leq \sup_{\|\tilde{v}\|_{Y^{-s}_{-\sigma,\Delta}}\leq 1} \int_{\mathbb{R}}(F(t),\tilde{v}(t))dt.
	\end{align*}
	\end{rem}

\begin{lemma} [\rm Bilinear Estimates \cite{Gu2017,GuReWa2021}]  \label{Bilinear}
Let $0<T<\infty$. Suppose that $\widehat{u}, \ \widehat{v}$ are localized in some compact intervals $I_1,I_2$ with $dist(I_1, I_2)\geq \lambda  >0$. Then for any $0<\varepsilon \ll 1$, we have
\begin{align}
\|u {v} \|_{L^2_{x, t\in[0,T]}}+  \|u\overline{v} \|_{L^2_{x, t\in[0,T]}} \lesssim (T^{\varepsilon/4} +T^{1/4}) \lambda^{-1/2 + \varepsilon} \|u\|_{V^2_\Delta} \|v\|_{V^2_\Delta}. \label{bilinear}
\end{align}
\end{lemma}
\begin{proof}
The proof follows \cite{Gu2017}.  However, \eqref{bilinear} admits arbitrarily large $T>0$, which is useful for our later purpose to consider the small initial data. We give the proof for the case $T>1$. Recall the bilinear Strichartz estimate
$$
\|u {v} \|_{L^2_{x, t\in[0,T]}} \lesssim \lambda^{-1/2} \|u\|_{U^2_\Delta} \|v\|_{U^2_\Delta}
$$
 By H\"older's and Strichartz' inequalities,
$$
\|u {v} \|_{L^2_{x, t\in[0,T]}} \leq T^{1/4} \|u   \|_{L^8_{ t\in[0,T]} L^4_x } \|v\|_{L^8_{ t\in[0,T]} L^4_x }  \lesssim  T^{1/4} \|u\|_{U^8_\Delta} \|v\|_{U^8_\Delta}.
$$
Applying the interpolation inequality in \cite{HaHeKo2009}, there exists $\epsilon >0$ such that for any $M>0$, there exists a decomposition $u=u_1+u_2$ with $u_1 \in U^2_\Delta$,  $u_2 \in U^8_\Delta$ verifying
$$
\frac{1}{M}\|u_1\|_{U^2_\Delta} +  e^{\epsilon M} \|u_2\|_{U^8_\Delta} \lesssim \|u \|_{V^2_\Delta}
$$
and similar decomposition holds for $v$. Hence, it follows from the bilinear and interpolation inequalities that
\begin{align*}
\|u {v} \|_{L^2_{x, t\in[0,T]}} & \leq  \|u_1 v_1\|_{L^2_{x, t\in[0,T]}} + \|u_1v_2 + u_2 v_1\|_{L^2_{x, t\in[0,T]}} + \|u_2 v_2\|_{L^2_{x, t\in[0,T]}} \\
 & \leq  \lambda^{-1/2}  \|u_1 \|_{U^2_\Delta}  \|v_1 \|_{U^2_\Delta}  + T^{1/4} \sum_{(i, j)\neq (1,1)} \|u_i\|_{U^8_\Delta } \| v_ j \|_{U^8_\Delta } \\
  & \lesssim  (M^2 \lambda^{-1/2}    + 2 T^{1/4} M e^{-\epsilon M} + T^{1/4}   e^{- 2\epsilon M} )  \|u \|_{V^2_\Delta } \| v \|_{V^2_\Delta }.
\end{align*}
Taking $M = \epsilon^{-1} \ln \langle\lambda \rangle$, we obtain that for $T>1$,
\begin{align*}
\|u {v} \|_{L^2_{x, t\in[0,T]}}
  & \lesssim    T^{1/4}  \lambda^{-1/2} \ln^2 \langle\lambda \rangle     \|u \|_{V^2_\Delta } \| v \|_{V^2_\Delta }.
\end{align*}
If $0<T<1$, one can find the proof in \cite{Gu2017}.
\end{proof}

	\begin{prop}[Linear estimate]\label{linearpart}
Let $s,\sigma \in \mathbb{R}$. Then we have
		\begin{equation}
			\|\chi_{[0,\infty)}(t)e^{it\Delta} u_0\|_{X^s_{\sigma,\Delta} }\sim \|u_0\|_{E^s_\sigma}.
		\end{equation}
	\end{prop}
	
	\subsection{Trilinear estimate }

For the sake of convenience, we write for any $M,N>0$,
\begin{align*}
& P_{< N}f = \mathscr{F}^{-1} \chi_{[0, N)}\mathscr{F} f , \ \ P_{\geq N}f = \mathscr{F}^{-1} \chi_{[N, \infty)}\mathscr{F} f,\\
& P_N f  = \mathscr{F}^{-1} \chi_{[N/2, N)}\mathscr{F} f , \ P_0 = P_{<1},  \ \   P_{M\leq \cdot < N}f  = \mathscr{F}^{-1} \chi_{[M, N)}\mathscr{F}f.
\end{align*}
For any $1\leq p\leq q \leq \infty$, the following Bernstein's estimates are well-known,
$$
\|P_{< N}f\|_{q} \lesssim N^{d(1/p-1/q)} \|P_{< N}f\|_{p}, \ \  \|P_{N}f\|_{q} \lesssim  N^{d(1/p-1/q)} \|P_{N}f\|_{p}.
$$

	\begin{prop}[Multi-linear estimate]\label{nonlinearpart}
	Let $T>0$, $s\leq 0, \, \sigma >-1/2$. Suppose that $u_j\in X^s_{\sigma,\Delta}$, $u_j(t) = 0 $ for $t<0$ or $t\geq T$, and ${\rm supp}\,\widehat{u_j(t)} \subset [0,\infty)$ for j=1,2,3, then we have for some $\varepsilon>0$,
	\begin{align*}
		\left\|\int_0^t e^{i(t-\tau)\Delta}(u_1u_2u_3^*)(\tau)d\tau\right\|_{X^s_{\sigma,\Delta}}\lesssim (T^\varepsilon +T) \|u_1\|_{X^s_{\sigma,\Delta}}\|u_2\|_{X^s_{\sigma,\Delta}}\|u_3\|_{X^s_{\sigma,\Delta}}.
	\end{align*}
	\end{prop}
	\begin{proof}[\textbf{Proof}]
	By duality, we only need to show
	\begin{align*}
		\left|\int_{\mathbb{R}^2}u_1u_2u_3^* \bar{v} dxdt\right|\lesssim (T^\varepsilon +T) \|u_1\|_{X^s_{\sigma,\Delta}}\|u_2\|_{X^s_{\sigma,\Delta}}\|u_3\|_{X^s_{\sigma,\Delta}} \|v\|_{Y^{-s}_{-\sigma,\Delta}}.
	\end{align*}
The proof for the cases $T\leq 1$ and $T>1$ are quite similar and we only consider the details for the case $T\leq 1$.	By density argument, we can assume that
	$\mathscr{F}_x u_j(t,\xi), j=1,2,3, \mathscr{F}_x v(t,\xi)$ have compact support. By Parserval's identity and $\widehat{uv} =\widehat{u}* \widehat{v}$, we have
	\begin{align*}
		\mathcal{L} &: =  \int_{\mathbb{R}^2} u_1u_2u^*_3 \bar{v} dxdt\\
		& =  \int_{[0,\infty)^3\times\mathbb{R}^3} \widehat{u}_1(\xi_1,\tau_1) \widehat{u}_2(\xi_2,\tau_2) \overline{\widehat{u}}_3 (\xi-\xi_1-\xi_2, \tau_1+\tau_2 - \tau) \overline{\widehat{ v}} (\xi, \tau)   d\xi_1d\xi_2d\xi d\tau_1d\tau_2d\tau.
	\end{align*}
If there is no confusion, $ [0,\infty)^3\times\mathbb{R}^3$ and $d\xi_1d\xi_2d\xi d\tau_1d\tau_2d\tau$ will be omitted in the expression of $\mathcal{L}$. Noticing that ${\rm supp} \mathscr{F}_x{u}_j, \mathscr{F}_x{v} \subset [0,\infty)$ we see that
	\begin{align*}
		\mathcal{L}
		 = &   \int  \widehat{P_{< 1}u}_1(\xi_1,\tau_1) \widehat{P_{< 1} u}_2(\xi_2,\tau_2) \overline{\widehat{P_{< 1} u}}_3 (\xi-\xi_1-\xi_2, \tau_1+\tau_2 - \tau) \overline{\widehat{P_{\leq 1} v}} (\xi, \tau) \\
  & \  +\sum_{N\in {  2^{\mathbb{N}}}}\int  \widehat{P_{< N}u}_1(\xi_1,\tau_1) \widehat{P_{< N } u}_2(\xi_2,\tau_2) \overline{\widehat{P_{< N } u}}_3 (\xi-\xi_1-\xi_2, \tau_1+\tau_2 - \tau) \overline{\widehat{P_N v}} (\xi, \tau)\\
   : = & \mathcal{L}_{lo} + \mathcal{L}_{hi}.
	\end{align*}
The estimate of $\mathcal{L}_{lo}$ is very easy. By H\"older's and Bernstein's inequalities,
\begin{align}
|\mathcal{L}_{lo}| & \lesssim T  \prod^3_{i=1}\|P_{< 1}u_i\|_{L^\infty_tL^2_x} \|P_{< 1 }v \|_{L^\infty_tL^2_x} \nonumber\\
& \lesssim T  \prod^3_{i=1}\|2^{s|\nabla|}\langle\nabla\rangle^\sigma P_{< 1 }u_i\|_{L^\infty_tL^2_x} \|2^{-s|\nabla|}\langle\nabla\rangle^{-\sigma} P_{< 1 }v \|_{L^\infty_tL^2_x}  \nonumber \\
&\lesssim T     \prod^3_{i=1}\|  u_i\|_{X^s_{\sigma, \Delta, T}}     \| v\|_{Y^{-s}_{-\sigma, \Delta, T} }. \label{Llow}
\end{align}
Now we estimate $\mathcal{L}_{hi}$.
Taking $\xi_3= \xi-\xi_1-\xi_2$ and $\tau_3 = \tau_1+\tau_2 - \tau$ in $\mathcal{L}_{hi}$,  we see that
\begin{align} \label{highestmodulation}
  \sum_{i=1,2,3} |\xi^2_i +\tau_i| + |\xi^2+\tau| \geq  2 (\xi_1+\xi_3)(\xi_2+\xi_3).
\end{align}
According to the lower bound of the highest modulation, one can decompose
	\begin{align*}
		\mathcal{L}_{hi}
		  = & \sum_{N\in {  2^{\mathbb{N}}}}\int  \widehat{P_{< N }u}_1(\xi_1,\tau_1) \widehat{P_{< N } u}_2(\xi_2,\tau_2) \overline{\widehat{P_{2^{-3}N\leq \cdot < N } u}}_3 (\xi_3, \tau_3) \overline{\widehat{P_N v}} (\xi, \tau)  \\
& + \sum_{N\in {  2^{\mathbb{N}}}}\int  \widehat{P_{< N }u}_1(\xi_1,\tau_1) \widehat{P_{< N } u}_2(\xi_2,\tau_2) \overline{\widehat{P_{\leq  2^{-3}N} u}}_3 (\xi_3, \tau_3) \overline{\widehat{P_N v}} (\xi, \tau)\\
: = & \mathcal{L}_1 +\mathcal{L}_2.
	\end{align*}

{\it Step 1.} We  estimate $\mathcal{L}_1$. By \eqref{highestmodulation}, the highest modulation of  ${P_{< N }u}_1$, ${P_{< N } u}_2$, $P_{2^{-3}N\leq \cdot < N } u_3$,  ${P_N v} $ satisfies
\begin{align} \label{highestmodulation1}
  \sum_{i=1,2,3} |\xi^2_i +\tau_i| + |\xi^2+\tau| \gtrsim N^2 .
\end{align}
First, we consider the case that ${P_{< N }u}_1$ has the highest modulation, i.e., $|\xi^2_1 +\tau_1|\geq N^2$. Noticing that $\xi=\xi_1+\xi_2+\xi_3$, we can rewrite $\mathcal{L}_1$ as
\begin{align*}
		\mathcal{L}_1
		  =   \sum_{N\in {  2^{\mathbb{N}}}}\int  2^{s \xi_1}\widehat{P_{< N }u}_1(\xi_1,\tau_1)  2^{s \xi_2}\widehat{P_{< N } u}_2(\xi_2,\tau_2)  2^{s \xi_3}\overline{\widehat{P_{2^{-3}N\leq \cdot < N } u}}_3 (\xi_3, \tau_3)  2^{-s \xi}\overline{\widehat{P_N v}} (\xi, \tau).
\end{align*}
By H\"older's inequality, we have
\begin{align}
|\mathcal{L}_1|
		  & \leq    \sum_{N\in { 2^{\mathbb{N}}}}  \|2^{s |\nabla|} P_{< N }u_1\|_{L^2_{x,t}}  \|2^{s |\nabla|} P_{< N } u_2\|_{L^4_tL^\infty_x}  \nonumber\\
&  \ \ \ \ \ \ \ \ \ \  \times \| 2^{s|\nabla|}   P_{2^{-3}N\leq \cdot < N } u_3 \|_{L^4_tL^\infty_x}  \|2^{-s |\nabla|} P_N v \|_{L^\infty_tL^2_x}. \label{L1est1}
\end{align}
Notice that $\sigma \in (-1/2, 0]$.  Using Bernstein's inequality, Proposition \ref{strichartz} and $U^2_\Delta \subset U^{4+}_\Delta $, for any $u $ with $u(t)=0$ for $t>T$, $0<\varepsilon \ll 1$,  one has that
\begin{align}
 \| P_{< N } u \|_{L^4_tL^\infty_x}
 & \lesssim  N^\varepsilon \| P_{< N } u\|_{L^4_t L^{1/\varepsilon}_x}   \nonumber\\
 & \lesssim  T^{\varepsilon/2}   N ^{\varepsilon-\sigma} \| \langle\nabla\rangle^{\sigma} P_{< N } u\|_{L^{4/(1-2\varepsilon)}_t L^{1/\varepsilon}_x} \nonumber \\
 & \lesssim   T^{\varepsilon/2}  N^{\varepsilon-\sigma} \| \langle\nabla\rangle^{\sigma}  u\|_{U^{4/(1-2\varepsilon)} _\Delta}\nonumber\\    & \lesssim   T^{\varepsilon/2} N^{\varepsilon-\sigma} \| \langle\nabla\rangle^{\sigma}  u\|_{U^2_\Delta}.
 \label{L1est2}
\end{align}
In view of the dispersion modulation decay estimates \eqref{dispersiondecay} and \eqref{highestmodulation1},
\begin{align}
  \|2^{s |\nabla|} P_{< N }u_1\|_{L^2_{x,t}}   \leq  N^{-1-\sigma} \|2^{s |\nabla|} \langle\nabla\rangle^\sigma  u_1\|_{V^2_{\Delta}}.  \label{L1est3}
\end{align}
Inserting the estimates of \eqref{L1est2} and \eqref{L1est3} into \eqref{L1est1}, we have
\begin{align}
|\mathcal{L}_1|
		  & \leq    \sum_{N\in {\{0\}\cup 2^{\mathbb{N}}}} T^\varepsilon  N^{-1+2\varepsilon-2\sigma} \|2^{s |\nabla|}\langle\nabla\rangle^\sigma u_1\|_{V^2_{\Delta}}   \|2^{s |\nabla|} \langle\nabla\rangle^\sigma u_2\|_{U^2_\Delta}  \nonumber\\
&  \ \ \ \ \ \ \ \ \ \  \times \| 2^{s|\nabla|} \langle\nabla\rangle^\sigma   u_3 \|_{U^2_\Delta}  \|2^{-s |\nabla|} \langle\nabla\rangle^{-\sigma} P_N v \|_{L^\infty_tL^2_x}. \label{L1est4}
\end{align}
 Take  $\varepsilon>0$ verifying $-1+ 2\varepsilon -2\sigma <0 $. Using the embedding $U^2\subset V^2_{-,rc} \subset L^\infty_tL^2_x$, it follows from \eqref{L1est4} that
 \begin{align}
|\mathcal{L}_1|
		  & \leq T^\varepsilon    \|2^{s |\nabla|}\langle\nabla\rangle^\sigma u_1\|_{V^2_{\Delta}}   \|2^{s |\nabla|} \langle\nabla\rangle^\sigma u_2\|_{U^2_\Delta}
  \| 2^{s|\nabla|} \langle\nabla\rangle^\sigma   u_3 \|_{U^2_\Delta}  \|2^{-s |\nabla|} \langle\nabla\rangle^{-\sigma}  v \|_{V^2_\Delta} \nonumber\\
  & \leq T^\varepsilon    \prod^3_{i=1}\|  u_i\|_{X^s_{\sigma, \Delta, T}}     \| v\|_{Y^{-s}_{-\sigma, \Delta, T} } \label{L1est5}
\end{align}
Next, if $ P_{ < N } u_2$ or $ P_{2^{-3}N\leq \cdot < N } u_3$ has the highest modulation, we can use the same way as in the above  to obtain \eqref{L1est5}.

Thirdly, we assume that $P_N v$ has the highest modulation. In order to use the highest modulation of  $P_N v$, $2^{-s|\nabla|} P_N v$ should be equipped with $L^2_{x,t}$ norm. Noticing that
$$
\|2^{-s|\nabla|} P_N v\|_{L^2_{x,t}} \sim N^{\sigma} \|2^{-s|\nabla|}\langle \nabla\rangle^{-\sigma} P_N v\|_{L^2_{x,t}},
$$
then we can repeat the procedures as in the above argument to get  \eqref{L1est5}.

{\it Step 2.} We consider the estimate of $\mathcal{L}_2$. Let us decompose $\mathcal{L}_2$ by
	\begin{align*}
\mathcal{L}_{2} = &  \sum_{N\in {  2^{\mathbb{N}}}}\int  \widehat{P_{< N }u}_1(\xi_1,\tau_1) \widehat{P_{\leq 2^{-3}N} u}_2(\xi_2,\tau_2) \overline{\widehat{P_{\leq  2^{-3}N} u}}_3 (\xi_3, \tau_3) \overline{\widehat{P_N v}} (\xi, \tau)\\
&+  \sum_{N\in {  2^{\mathbb{N}}}}\int  \widehat{P_{< N }u}_1(\xi_1,\tau_1) \widehat{P_{2^{-3} N\leq \cdot< N } u}_2(\xi_2,\tau_2) \overline{\widehat{P_{\leq  2^{-3}N} u}}_3 (\xi_3, \tau_3) \overline{\widehat{P_N v}} (\xi, \tau)\\
: = & \mathcal{L}_{21} +\mathcal{L}_{22}.
	\end{align*}
In $\mathcal{L}_{21}$,  $\xi =\xi_1+\xi_2+ \xi_3$ implies that $\xi_1\geq 2^{-2}N$, so one can rewrite $\mathcal{L}_{21}$ as
\begin{align*}
\mathcal{L}_{21} = &  \sum_{N\in {  2^{\mathbb{N}}}}\int  \widehat{P_{2^{-2}N\leq \cdot < N }u}_1(\xi_1,\tau_1) \widehat{P_{\leq 2^{-3}N} u}_2(\xi_2,\tau_2) \overline{\widehat{P_{\leq  2^{-3}N} u}}_3 (\xi_3, \tau_3) \overline{\widehat{P_N v}} (\xi, \tau)\\
= &  \sum_{N\in {  2^{\mathbb{N}}}}\int  2^{s\xi_1}\widehat{P_{2^{-2}N\leq \cdot < N }u}_1(\xi_1,\tau_1)  2^{s\xi_2}\widehat{P_{\leq 2^{-3}N} u}_2(\xi_2,\tau_2) 2^{s\xi_3} \overline{\widehat{P_{\leq  2^{-3}N} u}}_3 (\xi_3, \tau_3)  2^{-s\xi }\overline{\widehat{P_N v}} (\xi, \tau).
	\end{align*}
Applying H\"older's inequality and the bilinear estimate Lemma \ref{Bilinear}, we have for $-1-2\sigma + 2\varepsilon <0$,
\begin{align*}
|\mathcal{L}_{21}|  &  \leq   \sum_{N\in    2^{\mathbb{N} }}  \| 2^{s|\nabla|} P_{2^{-2}N\leq \cdot < N }u_1   2^{s|\nabla|}  P_{\leq 2^{-3}N} u_2 \|_{L^2_{x,t}} \| 2^{s|\nabla|}   P_{\leq  2^{-3}N} u_3 (-\cdot,\cdot)   2^{-s|\nabla|}  P_N v \|_{L^2_{x,t}} \\
 &  \leq  T^{\varepsilon/2} \sum_{N\in    2^{\mathbb{N} }} N^{-1+2\varepsilon}  \| 2^{s|\nabla|} P_{2^{-2}N\leq \cdot < N }u_1 \|_{V^2_\Delta}  \|2^{s|\nabla|}  P_{\leq 2^{-3}N} u_2 \|_{V^2_\Delta} \\
  & \ \ \ \  \times \| 2^{s|\nabla|}   P_{\leq  2^{-3}N} u_3 (-\cdot,\cdot)\|_{V^2_\Delta}  \| 2^{-s|\nabla|}  P_N v \|_{V^2_\Delta}\\
   &  \leq  T^{\varepsilon/2} \sum_{N\in    2^{\mathbb{N} }} N^{-1+2\varepsilon-2\sigma }  \| 2^{s|\nabla|}\langle \nabla\rangle^{\sigma}  u_1 \|_{V^2_\Delta}  \|2^{s|\nabla|} \langle \nabla\rangle^{\sigma}   u_2 \|_{V^2_\Delta} \\
  & \ \ \ \  \times \| 2^{s|\nabla|} \langle \nabla\rangle^{\sigma}    u_3 (-\cdot,\cdot)\|_{V^2_\Delta}  \| 2^{-s|\nabla|} \langle \nabla\rangle^{-\sigma}   v \|_{V^2_\Delta} \\
  &  \leq  T^{\varepsilon}
   \prod^3_{i=1}\|  u_i\|_{X^s_{\sigma, \Delta, T}}     \| v\|_{Y^{-s}_{-\sigma, \Delta, T} }.
	\end{align*}
For the estimate of $\mathcal{L}_{22}$, one can rewrite it as
\begin{align*}
\mathcal{L}_{22} = & \sum_{N\in {  2^{\mathbb{N}}}}\int  \widehat{P_{2^{-4}N \leq \cdot < N }u}_1(\xi_1,\tau_1) \widehat{P_{2^{-3} N\leq \cdot< N } u}_2(\xi_2,\tau_2) \overline{\widehat{P_{\leq  2^{-3}N} u}}_3 (\xi_3, \tau_3) \overline{\widehat{P_N v}} (\xi, \tau) \\
& + \sum_{N\in {  2^{\mathbb{N}}}}\int  \widehat{P_{ \leq 2^{-4}N }u}_1(\xi_1,\tau_1) \widehat{P_{2^{-3} N\leq \cdot< N } u}_2(\xi_2,\tau_2) \overline{\widehat{P_{\leq  2^{-3}N} u}}_3 (\xi_3, \tau_3) \overline{\widehat{P_N v}} (\xi, \tau)\\
:=& \mathcal{L}_{221}  +  \mathcal{L}_{222}.
\end{align*}
In $\mathcal{L}_{221}$, the highest modulation of $  P_{2^{-4}N \leq \cdot < N }u_1$, $ P_{2^{-3} N\leq \cdot< N } u_2$,
$ P_{\leq  2^{-3}N} u_3$ and $ P_N v$ satisfies
$$
\sum_{i=1,2,3} |\xi^2_i +\tau_i| + |\xi^2+\tau| \geq \xi_1\xi_2  \gtrsim N^2 .
$$
So, we can repeat the procedure as in Step 1 to obtain that
$$
|\mathcal{L}_{221}| \lesssim  T^{\varepsilon}
   \prod^3_{i=1}\|  u_i\|_{X^s_{\sigma, \Delta, T}}     \| v\|_{Y^{-s}_{-\sigma, \Delta, T} }.
$$
Noticing that $\mathcal{L}_{222}$ has similar structure and we can use the bilinear estimate Lemma \ref{Bilinear} to obtain that
$$
|\mathcal{L}_{222}| \lesssim  T^{\varepsilon}
   \prod^3_{i=1}\|  u_i\|_{X^s_{\sigma, \Delta, T}}     \| v\|_{Y^{-s}_{-\sigma, \Delta, T} }.
$$
So, we obtain that
\begin{align} \label{L2est}
|\mathcal{L}_{22}| \leq |\mathcal{L}_{221}| + |\mathcal{L}_{222}| \lesssim  T^{\varepsilon}
   \prod^3_{i=1}\|  u_i\|_{X^s_{\sigma, \Delta, T}}     \| v\|_{Y^{-s}_{-\sigma, \Delta, T} }.
\end{align}
Summarizing the estimates of $\mathcal{L}_{21}$ and $\mathcal{L}_{22}$, we have
\begin{align} \label{L2est1}
|\mathcal{L}_{2}|    \lesssim  T^{\varepsilon}
   \prod^3_{i=1}\|  u_i\|_{X^s_{\sigma, \Delta, T}}     \| v\|_{Y^{-s}_{-\sigma, \Delta, T} }.
\end{align}
By \eqref{Llow}, \eqref{L1est5}, \eqref{L2est1} and duality, we have the result for the case $T< 1 $, as desired. If $T>1$, we need to use $T^{1/4} \sim (T^{\varepsilon /4}+T^{1/4})$  in Lemma \ref{Bilinear} to substitute  $T^{\varepsilon /4 }$ in the above arguments and the details are omitted.
	\end{proof}

\subsection{Scalings}

\begin{lemma} \label{scaling}
		Assume $s\leq 0$, $\varphi \in E^s_\sigma (\mathbb{R}^d)$.  Write  $ D_\lambda: \, \varphi \to \varphi_\lambda (\cdot) = \varphi(\lambda\cdot)$. Assume that ${\rm supp }\, \widehat{\varphi} \subset \{\xi: |\xi|\geq \varepsilon_0\}$ for some $\varepsilon_0>0$. Then  for any $\lambda>1$, we have
\begin{align*}
& 			\|D_\lambda \varphi \|_{E^s_\sigma (\mathbb{R}^d) } \lesssim \lambda^{-d/2} 2^{s \lambda \varepsilon_0/2} \|\varphi \|_{E^{s}_\sigma (\mathbb{R}^d) }, \ \ \sigma \leq 0 ; \\
& 			\|D_\lambda \varphi \|_{E^s_\sigma (\mathbb{R}^d) } \lesssim \lambda^{-d/2+\sigma} 2^{s \lambda \varepsilon_0/2} \|\varphi \|_{E^{s}_\sigma (\mathbb{R}^d) }, \ \ \sigma > 0.
\end{align*}
For any $\mu<1$, we have
\begin{align*}
& \|D_\mu \varphi \|_{E^{s/\mu}_\sigma (\mathbb{R}^d) } \lesssim \mu^{-d/2+\sigma}  \|\varphi \|_{E^{s}_\sigma (\mathbb{R}^d) }, \ \ \sigma \leq 0; \\
& 	\|D_\mu \varphi \|_{E^{s/\mu}_\sigma (\mathbb{R}^d) } \lesssim \mu^{-d/2} \|\varphi \|_{E^{s}_\sigma (\mathbb{R}^d) }, \ \ \sigma > 0.
\end{align*}
\end{lemma}
	\begin{proof}[\textit{Proof}]
		By the definition of $E^s_\sigma (\mathbb{R}^d)$, we have
		\begin{align*}
			\|\varphi_\lambda\|_{E^s_\sigma}
			 = \lambda^{-d/2}\|\langle \lambda \xi\rangle^\sigma 2^{s\lambda |\xi|}\widehat{\varphi}(\xi)\|_{L^2_{\xi}}.
		\end{align*}
If  ${\rm supp } \, \widehat{\varphi}  \subset \{\xi: |\xi|\geq \varepsilon_0\}$ and $\sigma>0$, then
		\begin{align*}
			\|\varphi_\lambda\|_{E^s_\sigma }
			& \leq  \lambda^{-d/2 +\sigma }\|\langle   \xi\rangle^\sigma 2^{s\lambda |\xi|}\widehat{\varphi}(\xi)\|_{L^2_{\xi}}
			  \leq \lambda^{-d/2+\sigma } 2^{s (\lambda-1) \varepsilon_0} \|\varphi\|_{E^{s}_\sigma }.
		\end{align*}
If  ${\rm supp } \, \widehat{\varphi}  \subset \{\xi: |\xi|\geq \varepsilon_0\}$ and $\sigma\leq 0$, then
		\begin{align*}
			\|\varphi_\lambda\|_{E^s_\sigma }
			& \leq  \lambda^{-d/2 }\|\langle   \xi\rangle^\sigma 2^{s\lambda |\xi|}\widehat{\varphi}(\xi)\|_{L^2_{\xi}}
			  \leq \lambda^{-d/2} 2^{s (\lambda-1) \varepsilon_0} \|\varphi\|_{E^{s}_\sigma }.
		\end{align*}
Hence, we have the results for $\lambda >1$. For $0<\mu <1$, we have
\begin{align*}
			\|\varphi_\mu\|_{E^{s/\mu}_\sigma }
			 =   \mu^{-d/2 }\|\langle  \mu \xi\rangle^\sigma 2^{s |\xi|}\widehat{\varphi}(\xi)\|_{L^2_{\xi}}
		\end{align*}
Using $\langle  \mu \xi\rangle \geq \mu  \langle \xi\rangle $ for $\sigma \leq 0$, and  $\langle  \mu \xi\rangle \leq   \langle \xi\rangle $ for $\sigma > 0$, we have the results for $\mu <1$.
	\end{proof}

\begin{lemma}  \label{scalingxssigma}
Let $s, \sigma \leq 0$. Let $ u_\lambda (t,x) =u(\lambda^2t, \lambda x)$. We have
$$
\|u\|_{X^{s\lambda}_{\sigma,\Delta}} \leq  \lambda^{d/2-\sigma} \|u_\lambda\|_{X^{s}_{\sigma,\Delta}}, \ \ \lambda \gg 1.
$$

\end{lemma}

\begin{proof}
We can consider the case of spatial dimensions $d\geq 1$.  In view of $B_{\mathfrak{t}}(f_{\lambda^{-1}}, \, v) = \lambda^{d} B_{\lambda^{-2}\mathfrak{t}}(f, \, v_\lambda) $, we see that
$$
B (f_{\lambda^{-1}}, \, v) = \lambda^{d} B (f, \, v_\lambda).
$$
It follows that
$$
|B ( u, \, v) | \leq  \lambda^{d} \|u_\lambda\|_{X^s_{\sigma, \Delta}} \| v_\lambda\|_{Y^{-s}_{-\sigma, \Delta}}.
$$
Using the definition of the norm on $Y^{-s}_{-\sigma,\Delta}$,  we have
\begin{align*}
\| v_\lambda\|^2_{Y^{-s}_{-\sigma,\Delta} } &  =  \lambda^{-2d} \| \langle \xi\rangle^{-\sigma} 2^{- s|\xi|}  e^{ \mathrm{i} t|\xi|^2} \widehat{v(\lambda^2 t)} (\lambda^{-1} \xi )\|^2_{V^2} \\
 &   = \lambda^{-d }  \! \sup_{\{t_j\}^{J}_{j=0} \in \mathcal{Z}}  \sum^{J-1}_{j=0} \left\|  \langle \lambda \xi\rangle^{- \sigma}  2^{- s\lambda |\xi|}  [e^{ i t_j|\xi|^2} \widehat{v}( t_j, \xi) - e^{ i  t_{j-1}|\xi|^2} \widehat{v}( t_{j-1}, \xi) ] \right\|^2_{2}\\
  &   \leq  \lambda^{-d -2\sigma }  \! \sup_{\{t_j\}^{J}_{j=0} \in \mathcal{Z}}  \sum^{J-1}_{j=0} \left\|  \langle  \xi\rangle^{- \sigma}  2^{- s\lambda |\xi|}  [e^{ i t_j|\xi|^2} \widehat{v}( t_j, \xi) - e^{ i  t_{j-1}|\xi|^2} \widehat{v}( t_{j-1}, \xi) ] \right\|^2_{2}
\end{align*}
It follows that
\begin{align*}
\| v_\lambda\|_{Y^{-s}_{-\sigma,\Delta}} \lesssim \lambda^{-d/2-\sigma }  \| v \|_{Y^{-s\lambda}_{-\sigma,\Delta}}.
\end{align*}
By duality, we get the result.
\end{proof}

\subsection{Proof of Theorem \ref{mainresult}}

	\begin{proof}[\textit{Proof of Theorem \ref{mainresult}}]
		First, we show the local well-posedness and let us consider the mapping
		\begin{align*}
			\mathscr{T}: u\mapsto \chi_{[0,\infty)}(t)e^{it\Delta} u_0 -i\int_0^t e^{i(t-\tau)\Delta}(\chi_{[0,T)}u^2u^*)(\tau)d\tau.
		\end{align*}
By Proposition \ref{linearpart} and \ref{nonlinearpart}, we have for $0<T<1$,
$$
\|\mathscr{T}u\|_{X^s_{\sigma, \Delta}} \lesssim \|u_0\|_{E^s_\sigma }+T^\varepsilon \|u\|^3_{X^s_{\sigma, \Delta}}
$$
and
$$
\|\mathscr{T}u-\mathscr{T} v\|_{X^s_{\sigma, \Delta}}\lesssim T^\varepsilon (\|u\|_{X^s_{\sigma, \Delta}}^2+\|v\|_{X^s_{\sigma, \Delta}}^2)\|u-v\|_{X^s_{\sigma, \Delta}}.
$$
Applying the contraction mapping principle, we can obtain that
		\begin{align} \label{NNLSINT}
			  u(t) = \chi_{[0,\infty)}(t)e^{it\Delta} u_0 -i\int_0^t e^{i(t-\tau)\Delta}(\chi_{[0,T)}u^2u^*)(\tau)d\tau.
		\end{align}
has a unique solution  $u\in X^s_{\sigma, \Delta}$ if
\begin{align}
T = \frac{1}{(2C\|u_0\|^2_{E^s_\sigma})^{1/\varepsilon}},  \label{Tlocal}
\end{align}
and the solution satisfies
\begin{align}
\|u\|_{L^\infty(0, T ; E^s_\sigma) \cap X^s_{\sigma, \Delta}} \leq  2C \|u_0\|_{E^s_\sigma}.   \label{Tlocal2}
\end{align}
If $\|u_0\|_{E^s_\sigma}$ is sufficiently small, one can replace $T^{\varepsilon/2}$ by $T$ in the above arguments and take
\begin{align}
T =  \frac{1}{ 2C\|u_0\|^2_{E^s_\sigma} } >1.  \label{Tlocalsmall}
\end{align}
By Proposition \ref{linearpart} and \ref{nonlinearpart}, we have
$$
\|\mathscr{T}u\|_{X^s_{\sigma, \Delta}} \lesssim \|u_0\|_{E^s_\sigma }+T  \|u\|^3_{X^s_{\sigma, \Delta}}.
$$
Applying the same way as above we can get the local solution $u\in X^s_{\sigma, \Delta}$ and \eqref{Tlocal2} also holds. If $u$ satisfies \eqref{NNLSINT}, then $u$ solves
		\begin{align} \label{NNLSINTori}
			  u(t) =  e^{it\Delta} u_0 -i\int_0^t e^{i(t-\tau)\Delta}( u^2u^*)(\tau)d\tau
		\end{align}
in the time interval $[0,T)$.

Next, we consider the global existence and uniqueness of solutions. We assume, without loss of generality that $\mathrm{supp}~\widehat{u}_0\subset [\varepsilon_0,\infty)$ for some $\varepsilon_0 >0$. Denote $u_{0,\lambda}(x) = \lambda u_0(\lambda x)$.
		Following Lemma \ref{scaling},  we have
\begin{align*}
			 \|u_{0,\lambda}\|_{E^{s}_\sigma }  \lesssim \lambda^{1/2} 2^{s \lambda \varepsilon_0/2} \|u_0 \|_{E^{s}_\sigma } \leq 2^{s \lambda \varepsilon_0/4} , \ \ \lambda \geq \lambda_0:= \lambda_0 (\sigma, \varepsilon_0, \|u_0\|_{E^{s}_\sigma }).
\end{align*}
By the local well-posedness, we see that NNLS \eqref{NNLSINT} with initial data $u_{0,\lambda}$  has a unique solution
$$
u_\lambda \in L^\infty(0, T_\lambda; E^s_\sigma ) \cap  X^s_{\sigma, \Delta} ,  \ \ T_\lambda =   1/     2C \|u_{0,\lambda} \|^2_{E^{s}_\sigma } ,
$$
$$
\|u_\lambda\|_{L^\infty(0, T_\lambda; E^s_\sigma) \cap  X^s_{\sigma, \Delta} }  \leq  2 C  \|u_{0,\lambda} \|_{E^{s}_\sigma } .
$$
By choosing $\lambda \geq \lambda_1:= \lambda_1(\sigma, \varepsilon_0,  \|u_0\|_{E^{s}_\sigma })$, we see that
$$
T_\lambda \geq 2^{\sqrt{\lambda}}, \ \ \  \|u_\lambda\|_{L^\infty(0, T_\lambda; E^s_\sigma) \cap X^s_{\sigma, \Delta} }  \leq  2^{- \sqrt{\lambda}}.
$$
Then, one sees that
\begin{align} \label{scalinguu}
 u(t, x) = \lambda^{-1} u_\lambda(\lambda^{-2}  t,\lambda^{-1}  x)
\end{align}
is the solution of \eqref{NNLS} with initial data $u_{0}$. By Lemma \ref{scaling}, for any $t\leq   2^{ \sqrt{\lambda}  } $ and $\lambda\geq \lambda_1$,
$$
\| u(t)\|_{E^{s\lambda}_\sigma } = \lambda^{-1} \|u_\lambda(\lambda^{-2}  t,\lambda^{-1} \cdot )\|_{E^{s\lambda}_\sigma } \leq   C \lambda^{-1/2-\sigma}  2^{ - \sqrt{\lambda } }\leq    2^{ - \sqrt{\lambda } } .
 $$
In view of Lemma \ref{scalingxssigma}, we see that $u\in X^{s\lambda}_{\sigma, \Delta}$ if $\lambda\geq \lambda_1$.
Now, taking $\lambda =j \in \mathbb{N}$ and $j\geq \lambda_1 $, we obtain a unique solution of NNLS \eqref{NNLSINT}
$
u\in L^\infty(0,   2^{\sqrt{ j}} ; E^{sj}_\sigma ) \cap X^{sj}_{\sigma, \Delta}
$
satisfying
$$
\| u \|_{L^\infty(0,   2^{ \sqrt{j} } ; E^{sj}_\sigma) \cap  X^{sj}_{\sigma, \Delta} }   \leq     2^{-\sqrt{j}}.
 $$
Since $1\ll j\in \mathbb{N}$ is arbitrary, we see that the solution is a global one.   So, we can finish the proof of Theorem \ref{mainresult}.
 \end{proof}

 \subsection{Ill-posedness} \label{Illposed}

 We show that NNLS is ill-posed in any $E^{s'}_{\sigma'}$ if $u_0 \in E^s_{\sigma}$ ($s'\leq s<0, \, \sigma'\leq \sigma, \, \sigma >-1/2$) does not satisfy the condition ${\rm supp} \ \widehat{u}_0 \subset [0,\infty)$, where the ill-posedness means that the solution map from $E^s_{\sigma}$ into $E^{s'}_{\sigma'}$ is not $C^3$. In fact, let\footnote{We denote by $\chi_E$ the characteristic function on $E$.}
 $$
 \widehat{\varphi}(\xi) = 2^{-sk/2} (\chi_{k+I}(\xi) + \chi_{-2k+2I}(\xi) ),  \ \ I =[1/8,\, 1/4], \ 2I = [1/4, 1/2], \ \ k\gg 1.
 $$
 Put $u_0 =\varepsilon \varphi$ for $0<\varepsilon \ll 1$. It is easy to see that $u_0$ is smooth and
 $$
 \|u_0\|_{E^s_\sigma}  \lesssim  \varepsilon 2^{sk/4}.
 $$
If \eqref{NNLS} has a solution $u_\varepsilon$ and the solution map $\varepsilon \varphi \to u_\varepsilon$ is $C^3$, then we have
\begin{align}
& \frac{\partial u_\varepsilon(t)}{\partial \varepsilon}|_{\varepsilon=0} =  e^{it\partial^2_x} \varphi, \ \  \frac{\partial^2 u_\varepsilon(t)}{\partial \varepsilon^2}|_{\varepsilon=0} =0, \label{illposed1} \\
& \frac{\partial^3 u_\varepsilon(t)}{\partial\varepsilon^3 }|_{\varepsilon=0} =\mathrm{i}6 \alpha \int^t_0 e^{\mathrm{i}(t-\tau)\partial^2_x} (e^{\mathrm{i}\tau \partial^2_x} \varphi)  (e^{\mathrm{i}\tau \partial^2_x} \varphi) (e^{\mathrm{i}\tau \partial^2_x} \varphi)^* d\tau.  \label{illposed2}
\end{align}
Noticing that $\frac{\widehat{\partial^3 u_\varepsilon(t)} (\xi)}{\partial\varepsilon^3 }|_{\varepsilon=0} $ is supported in $[1/2,1]\cup (3k+ [1/2,1]) \cup (-6k+[1/2,1]) \cup (-3k+[1/2,1]) $, one can easily see that for $k\gg 1$,
\begin{align}
& \left|\frac{\widehat{\partial^3 u_\varepsilon(t)} (\xi)}{\partial\varepsilon^3 }|_{\varepsilon=0}  \right| \nonumber  \\
& \ \ \ \gtrsim \left| \chi_{[1/2,1]}(\xi)\int_{\mathbb{R}^2}  \frac{e^{\mathrm{i}2 t (\xi-\xi_1)(\xi-\xi_2)} -1 }{2(\xi-\xi_1)(\xi-\xi_2) }  \varphi(\xi_1) \varphi(\xi_2) \varphi(\xi-\xi_1-\xi_2)   d\xi_1d\xi_2 \right|:= R(\xi).\label{illposed3}
\end{align}
We can rewrite $\chi_{[1/2,1]}(\xi)\varphi(\xi_1) \varphi(\xi_2) \varphi(\xi-\xi_1-\xi_2) $ as
\begin{align}
\chi_{[1/2,1]}(\xi)
    &  \varphi(\xi_1) \varphi(\xi_2) \varphi(\xi-\xi_1-\xi_2)  \nonumber\\
= & \ 2^{-3sk/2} \chi_{[1/2,1]}(\xi)    \chi_{k+I}(\xi_1) \chi_{k+I}(\xi_2) \chi_{-2k+2I}(\xi-\xi_1-\xi_2)  \nonumber\\
 &  +  2^{-3sk/2} \chi_{[1/2,1]}(\xi)    \chi_{-2k+2I}(\xi_1) \chi_{k+I}(\xi_2) \chi_{k+I}(\xi-\xi_1-\xi_2)  \nonumber\\
 & +  2^{-3sk/2} \chi_{[1/2,1]}(\xi)    \chi_{k+I}(\xi_1) \chi_{-2k+2I}(\xi_2) \chi_{k+I}(\xi-\xi_1-\xi_2).
 \label{illposed4}
\end{align}
It follows that
\begin{align}
R(\xi)\geq  2^{-sk}\left| \chi_{[1/2,1]}(\xi)\int_{\mathbb{R}^2} \varrho(t, \xi,\xi_1,\xi_2)      \chi_{k+I}(\xi_1) \chi_{k+I}(\xi_2) \chi_{-2k+2I}(\xi-\xi_1-\xi_2)   d\xi_1d\xi_2 \right|.\label{illposed5}
\end{align}
where
\begin{align}
  \varrho(t, \xi,\xi_1,\xi_2)=  \frac{e^{\mathrm{i}2 t (\xi-\xi_1)(\xi-\xi_2)} -1 }{2(\xi-\xi_1)(\xi-\xi_2) }  - \frac{e^{\mathrm{i}2 t (\xi-\xi_1)(\xi_1+\xi_2)} -1 }{ (\xi-\xi_1)(\xi_1+\xi_2) }    \label{illposed6}
\end{align}
Taking $t  = \kappa /k^2, \, 0<\kappa\ll 1$, we see that
\begin{align}
 - Im \, \varrho(t, \xi,\xi_1,\xi_2) & = \frac{\sin 2 t (\xi-\xi_1)(\xi_1+\xi_2) }{ (\xi-\xi_1)(\xi_1+\xi_2) } - \frac{\sin 2 t (\xi-\xi_1)(\xi-\xi_2) }{2(\xi-\xi_1)(\xi-\xi_2) }  \nonumber\\
 &= t(1+o(\kappa)) \geq t/2.
     \label{illposed7}
\end{align}
From \eqref{illposed5}--\eqref{illposed7} it follows that for $t=\kappa/k^2$,
\begin{align}
R(\xi) \gtrsim 2^{-sk} \frac{\kappa}{k^2} \left| \chi_{[1/2,1]}(\xi)\int_{\mathbb{R}^2}    \chi_{ I}(\xi_1) \chi_{ I}(\xi_2) \chi_{ 2I}(\xi-\xi_1-\xi_2)   d\xi_1d\xi_2  \right| \gtrsim  2^{-sk} \frac{\kappa}{k^2}  \chi_{[5/8,7/8]}(\xi)  .\label{illposed8}
\end{align}
  Hence, for any $\sigma', s'\in \mathbb{R}$, we have for $t=\kappa/k^2$,
\begin{align}
  \left\|\frac{\widehat{\partial^3 u_\varepsilon(t)} (\xi)}{\partial\varepsilon^3 }|_{\varepsilon=0}  \right\|_{E^{s'}_{\sigma'}}
 \gtrsim  2^{-sk/2}  \label{illposed9}
\end{align}
So, we obtain that $\frac{\partial^3 u_\varepsilon(t)}{\partial\varepsilon^3 }$ in any $E^{s'}_{\sigma'}$ is not continuous at $\varepsilon =0$.

\section{Global solution of NdNLS} \label{GloNDNLS}

In this section we prove the global existence and uniqueness of the nonlocal derivative NLS.

\subsection{Equivalent equation via gauge transformation}

Using the following nonlocal gauge transform
\begin{align} \label{gauge1}
v(t,x) =\mathcal{G}(u)  = u(t,x)\exp\left(- \frac{\delta}{2} \left(\int^x_{-\infty} - \int_x^{\infty}\right) u(t,y) {u}^*(t,y)dy \right)
                       := u \mathcal{E}(u),
\end{align}
we can show that NdNLS is equivalent to the following equation
\begin{equation}\label{NdNLSequiv}
		 \mathrm{i} v_t + \partial^2_x v - \alpha\,  v^2 \partial_x v^* -\frac{\alpha^2}{2} v^3 (v^*)^2=0,  \  v(x,0) = v_0(x).
\end{equation}

\begin{prop} \label{Equivndnls}
If $u$ is a smooth solution of NdNLS \eqref{NdNLS}, then $v=\mathcal{G}(u)$ with $\alpha =-2\delta$ is the smooth solution of \eqref{NdNLSequiv} with initial data $v_0= \mathcal{G}(u_0)$. Conversely, if $v$ is a smooth solution of \eqref{NdNLSequiv}, then $u= v \mathcal{E}^{-1}(v)$ with initial data $u_0 = v_0 \mathcal{E}^{-1}(v_0)$ is a smooth solution of \eqref{NdNLS}.
\end{prop}
\begin{proof}
It is easy to see that for $v=\mathcal{G}(u)$
$$
v^*(t,x)  = u^*(t,x) \exp\left( \frac{\delta}{2} \left(\int^x_{-\infty} - \int_x^{\infty}\right) u(t,y) {u}^*(t,y)dy \right).
$$
It follows that
$$
v v^* = \mathcal{G}(u) \mathcal{G}(u)^* = u u^*.
$$
So, the gauge transform \eqref{gauge1} can be rewritten as
\begin{align} \label{gauge2}
u(t,x) = v(t,x)\exp\left(\frac{\delta}{2} \left(\int^x_{-\infty} - \int_x^{\infty}\right) v(t,y) {v}^*(t,y)dy \right) = v\mathcal{E}^{-1} (v),
\end{align}
Using \eqref{gauge2}, one can calculate that
$$
u_x = (v_x + \delta v^2 v^*) \mathcal{E}^{-1}(v)
$$
and
\begin{align} \label{gauge3}
u_{xx} & = [v_{xx} + 3\delta v v_x v^* + \delta v^2 v^*_x + \delta^2 (vv^*)^2 v] \mathcal{E}^{-1} (v), \\
u_{t} & = \left[v_{t} + \frac{\delta}{2} v \left(\int^x_{-\infty} - \int_x^{\infty}\right) \partial_t (vv^*)(y)dy \right] \mathcal{E}^{-1} (v). \label{gauge4}
\end{align}
From \eqref{NdNLS} it follows that
\begin{align} \label{gauge5}
\partial_t(uu^*) & = \mathrm{i}\, \partial_x [ u_x u^* - u^*_x u + \frac{\alpha}{2} (uu^*)^2 ].
\end{align}
Inserting \eqref{gauge5} into \eqref{gauge4} and noticing that $uu^*=vv^*$,  we have
\begin{align}
\mathrm{i}\, u_{t} & = \left[\mathrm{i}\, v_{t} - \delta vv_x v^* + \delta v^2 v^*_x - \delta (2\delta + \alpha/2) (vv^*)^2 v \right] \mathcal{E}^{-1} (v). \label{gauge6}
\end{align}
One easily sees that
\begin{align} \label{gauge7}
\alpha uu_x u^* & = (\alpha vv_x v^* + \alpha\delta (vv^*)^2 v) \mathcal{E}^{-1} (v).
\end{align}
Summarizing  \eqref{gauge3}, \eqref{gauge6}, \eqref{gauge7}, we have
\begin{align}
& \mathrm{i} u_t + \partial^2_x u + \alpha\,  u u_x u^* \nonumber\\
& = [ \mathrm{i} v_t + \partial^2_x v + (2\delta+\alpha)  v v_x v^* + 2\delta v^2 \partial_x v^* +(\alpha\delta/2 -\delta^2) v^3 (v^*)^2 ] \mathcal{E}^{-1} (v).
\end{align}
Taking $\alpha+ 2\delta =0$, we obtain \eqref{NdNLSequiv}. Using an analogous way as above, we can show that $u= v \mathcal{E}^{-1}(v)$ with initial data $u_0 = v_0 \mathcal{E}^{-1}(v_0)$ is a smooth solution of \eqref{NdNLS} if $v$ is a smooth solution of \eqref{NdNLSequiv}, the details are omitted.
\end{proof}

\subsection{Boundness of gauge transform in $E^s_\sigma$}

Let us recall the definition of Besov spaces $B^\sigma_{p,q}$, cf. \cite{BeLo1976,Tr1983}.  Let $\psi: \mathbb{R} \rightarrow
[0, 1]$ be a smooth cut-off function which equals $1$ on  $[-1,1]$ and equals $0$ outside  $[-2,2]$.
Write
\begin{align} \label{phi}
\varphi(\xi):=\psi(\xi)-\psi(2\xi),  \ \ \varphi_j(\xi)=\varphi(2^{-j}\xi).
\end{align}
$\triangle_j:=\mathscr{F}^{-1}\varphi_j \mathscr{F}, \  j\in \mathbb{N}$ and $\triangle_0:=\mathscr{F}^{-1}\psi \mathscr{F}$
are said to be the dyadic decomposition operators. One easily sees that ${\rm supp} \varphi_j \subset [2^{j-1} 2^{j+1}] \cup [-2^{j+1}, -2^{j-1}]  \, j\in \mathbb{N}$.
Let $\sigma \in \mathbb{R}$, $1\le p,q\le \infty$.  The norm on Besov spaces are defined as follows:
\begin{align}
\label{Besov} \|f\|_{{B}^s_{p, q}}=
 \left(\sum_{j=0}^{+\infty}2^{jsq}\|\triangle_j f\|^q_{p}\right)^{1/q}
\end{align}
with the usual modification for $q=\infty$.

\begin{prop} \label{BDgauge}
Let $s\leq 0$, $\sigma >0$. Suppose that $u\in E^s_\sigma$ satisfying ${\rm supp} \widehat{u} (\xi) \subset \{\xi: \ \xi\geq \varepsilon_0\}$ for some $\varepsilon_0 >0$. Then we have $\mathcal{G}(u) \in E^s_\sigma$ and there exists $C:= C(\varepsilon_0,\delta,\sigma)>0$ such that
\begin{align}
\|\mathcal{G}(u)\|_{E^s_\sigma} \leq  \exp (C \|u\|^2_{E^s_\sigma}) \|u\|_{E^s_\sigma}. \label{bdgauge1}
\end{align}
\end{prop}
\begin{proof}
Let us recall that
$
\partial^{-1}_x f = \frac{1}{2}\left(\int^x_{-\infty} + \int^x_{+\infty}\right) f(y) dy.
$
By Taylor's expansion, we have
$$
\mathcal{G}(u) = u \sum^\infty_{k=0} \frac{(-\delta)^k}{k!} (\partial^{-1}_x(uu^*))^k.
$$
It follows that
\begin{align}
\|\mathcal{G}(u)\|_{E^s_\sigma} \leq  \sum^\infty_{k=0} \frac{|\delta|^k}{k!} \|u(\partial^{-1}_x(uu^*))^k\|_{E^s_\sigma}. \label{bdgauge2}
\end{align}
Let us observe that, if ${\rm supp}\, \widehat{w}_j \subset [0, \infty)$, then
\begin{align}
2^{s|\nabla|} (w_1w_2...w_J) = \prod^J_{j=1} (2^{s|\nabla|} w_j). \label{bdgauge3}
\end{align}
For simply, we further write $2^{s|\nabla|} u =u_1 , \ 2^{s|\nabla|} u =u_2$ Using the definition of $E^s_\sigma$ and \eqref{bdgauge3},
\begin{align}
\|u(\partial^{-1}_x(uu^*))^k\|_{E^s_\sigma} = \|u_1 (\partial^{-1}_x(u_1u_2))\|_{H^\sigma}.
\end{align}
Recall that for any $\sigma \in (0,1)$, $1/p=1/p_1+1/p_2= 1/r_1 +1/r_2$ (cf. \cite{WaHuHaGu2011}),
\begin{align} \label{Besovest1}
\|fg\|_{B^\sigma_{p,q}} \lesssim \|f\|_{B^\sigma_{r_1,q}} \|g\|_{r_2} + \|f\|_{p_1}\|g\|_{B^\sigma_{p_2,q}}.
\end{align}
By taking $p=q=r_1=2$, $r_2=\infty$, we have from $H^\sigma = B^\sigma_{2,2}$ that
$$
\|fg\|_{H^\sigma} \lesssim \|f\|_{H^\sigma} \|g\|_{\infty} + \|f\|_{p_1}\|g\|_{B^\sigma_{p_2,2}},
$$
where $p_2= 1/\sigma, \, 1/p_1 = 1/2-\sigma$. It follows that
\begin{align}
\|u_1 (\partial^{-1}_x(u_1u_2))^k\|_{H^\sigma} \lesssim \|u_1\|_{H^\sigma} \|\partial^{-1}_x(u_1u_2)\|^k_{\infty} + \|u_1\|_{p_1}\|(\partial^{-1}_x(u_1u_2))^k\|_{B^\sigma_{p_2,2}}=I+II.
\end{align}
By Young's inequality,
\begin{align}
I & \lesssim \|u\|_{E^s_\sigma} \|\xi^{-1}(\widehat{u}_1 * \widehat{u}_2)\|^k_{L^1_\xi}  \nonumber\\
& \lesssim \|u\|_{E^s_\sigma} \|\xi^{-1-\sigma}\chi_{\{\xi\geq 2\varepsilon_0\}}\|^k_{L^1} \|\xi^\sigma (\widehat{u}_1 * \widehat{u}_2)\|^k_{L^\infty_\xi}  \nonumber\\
& \leq C^k_{\varepsilon_0,\sigma} \|u\|_{E^s_\sigma} (\|u_1\|_{H^\sigma}\|u_2\|_2 + \|u_1\|_2\|u_2\|_{H^\sigma} )^k \nonumber\\
& \leq C^k_{\varepsilon_0,\sigma} \|u\|^{2k+1}_{E^s_\sigma}. \label{xbdgauge3}
\end{align}
Next, we estimate $II$. It suffices to consider the case $\sigma <1/2$. By Sobolev embedding, we have
$$
\|u_1\|_{p_1} \lesssim \|u_1\|_{H^\sigma} = \|u\|_{E^s_\sigma}.
$$
From \eqref{Besovest1} it follows that
\begin{align}
\|(\partial^{-1}_x(u_1u_2))^k\|_{B^\sigma_{p_2,2}} & \lesssim \|\partial^{-1}_x(u_1u_2)\|_{B^\sigma_{p_2,2}} \|\partial^{-1}_x(u_1u_2)\|^{k-1}_{\infty} \nonumber\\
& \lesssim C^{2(k-1)}_{\varepsilon_0,\delta} \|u\|^{2(k-1)}_{E^s_\sigma} \|\partial^{-1}_x(u_1u_2)\|_{B^\sigma_{p_2,2}}. \label{Besovest2}
\end{align}
Using the embedding $B^{\sigma+\varepsilon}_{p_2, \infty} \subset B^\sigma_{p_2,2}$ (cf. \cite{Tr1983}), we see that for any $0<\varepsilon<\sigma$,
\begin{align}
\|\partial^{-1}_x(u_1u_2)\|_{B^\sigma_{p_2,2}} & \lesssim \|\partial^{-1}_x(u_1u_2)\|_{B^{\sigma+\varepsilon}_{p_2,\infty}} \nonumber\\
& \lesssim \|(I-\Delta)^{(-1+\sigma+\varepsilon)/2} (u_1u_2)\|_{p_2} \nonumber\\
& \lesssim \|\langle \xi\rangle^{-1+\sigma+\varepsilon} (\widehat{u}_1* \widehat{u}_2)\|_{p'_2} \nonumber\\
& \lesssim \|\langle \xi\rangle^{-1+\varepsilon}\|_{p'_2}\| \langle \xi\rangle^{\sigma}(\widehat{u}_1* \widehat{u}_2)\|_{\infty} \nonumber\\
 & \leq C_{\varepsilon_0,\sigma} \|u\|^{2}_{E^s_\sigma}. \label{Besovest3}
\end{align}
By \eqref{Besovest2} and \eqref{Besovest3},
$$
II \leq C^{k}_{\varepsilon_0,\sigma} \|u\|^{2k+1}_{E^s_\sigma}.
$$
Summarizing the estimates of $I$ and $II$, we have
\begin{align} \label{Besovest4}
\| u(\partial^{-1}_x(uu^*))^k \|_{E^s_\sigma} \leq C^{k}_{\varepsilon_0,\sigma} \|u\|^{2k+1}_{E^s_\sigma}.
\end{align}
Inserting the estimate \eqref{Besovest4} into \eqref{bdgauge2}, we immediately have the result, as desired.
\end{proof}

\begin{rem}
Using the same way as in \eqref{bdgauge1}, we can show that
\begin{align}
\|\mathcal{G}(u) -\mathcal{G}(v)\|_{E^s_\sigma} \leq  (\exp (C \|u\|^2_{E^s_\sigma})+ \exp (C \|u\|^2_{E^s_\sigma}) ) \|u-v\|_{E^s_\sigma}  \label{rembdgauge1}
\end{align}
if $u,\, v$ satisfy the conditions of Proposition \ref{BDgauge}.  In fact, using Taylor's expansion, we have
\begin{align*}
\mathcal{G}(u)- \mathcal{G}(v) = & (u-v) \sum^\infty_{k=0} \frac{\delta^k}{k!} (\partial^{-1}_x(uu^*))^k
 +   v \sum^\infty_{k=1} \frac{\delta^k}{k!}  \partial^{-1}_x[(u u^*)^k- (vv^*)^k] \\
 = & (u-v) \sum^\infty_{k=0} \frac{\delta^k}{k!} (\partial^{-1}_x(uu^*))^k \\
 & +   v \sum^\infty_{k=1} \frac{\delta^k}{k!}  \partial^{-1}_x \left( u^k(u^*-v^*) \sum^{k-1}_{i=1}(u^*)^{k-i-1}(v^*)^i + (v^*)^k (u-v)\sum^{k-1}_{i=1} u^{k-i-1}v^i \right),
\end{align*}
where we assume that $\sum^0_{i=1} a_i=1$.  Then one can repeat the proof of \eqref{bdgauge1} to get \eqref{rembdgauge1}.
\end{rem}

\subsection{Multi-linear estimates}
\label{Multidnls}

For convenience, we introduce the following
 \begin{align*}
&		\|u\|_{\mathcal{X}^s_{\sigma,\Delta}}: = \left(\sum_{N\in 2^{\mathbb{N}} \cup \{0\}}  \langle N\rangle^{2\sigma} \| 2^{s|\nabla|} P_N  u\|^2_{U^2_{\Delta}} \right)^{1/2}, \\
&		\|v\|_{\mathcal{Y}^s_{\sigma,\Delta}}: = \left(\sum_{N\in 2^{\mathbb{N}} \cup \{0\}}  \langle N\rangle^{2\sigma} \| 2^{s|\nabla|} P_N  v \|^2_{V^2_{\Delta}} \right)^{1/2}.
\end{align*}
For any $T>0$, define
\begin{align}
\mathcal{X}^s_{\sigma,\Delta,T}:=\{u\in \mathcal{X}^s_{\sigma,\Delta} :\ u(t) = e^{i(t-T)\Delta}u(T), \ \forall~ t\geq T\}. \label{Xssigmadelta}
\end{align}
\begin{prop}
Let $s\leq 0, \, \sigma \in \mathbb{R}$. Then $\mathcal{X}^s_{\sigma,\Delta}$ and  $\mathcal{Y}^{-s}_{-\sigma,\Delta}$ are dual spaces and
\begin{equation}\label{ddualuse}
			\|u\|_{\mathcal{X}^s_{\sigma,\Delta}} = \sup_{\|v\|_{\mathcal{Y}^{-s}_{-\sigma,\Delta}}\leq 1 }
		\int_{\mathbb{R}} (F(t), v(t))dt.
		\end{equation}
Moreover, we have the following inclusion relations
\begin{align} \label{Inclusion}
\mathcal{Y}^s_{\sigma,\Delta} \subset {Y}^s_{\sigma,\Delta}, \ \  {X}^s_{\sigma,\Delta} \subset \mathcal{X}^s_{\sigma,\Delta}.
\end{align}
\end{prop}
\begin{proof}
The duality between $\mathcal{X}^s_{\sigma,\Delta}$ and  $\mathcal{Y}^{-s}_{-\sigma,\Delta}$ follows an analogous way as in  \eqref{dualuse}.   The first inclusion in \eqref{Inclusion} can be obtained by using the definitions  of ${Y}^s_{\sigma,\Delta}$ and $\mathcal{Y}^s_{\sigma,\Delta}$.  In view of the first inclusion and the duality, we have the second embedding in \eqref{Inclusion}.
\end{proof}

Now,  let us consider the quintilinear estimate in $\mathcal{X}^s_{\sigma,\Delta}$.
\begin{prop}[Quintilinear estimate]\label{dquintilinear}
	Let $T>0$, $s\leq 0$, $\sigma>0$. Assume that $u_j\in \mathcal{X}^s_{\sigma,\Delta}$, $u_j(t) = 0 $ for $t<0$ or $t\geq T$, and ${\rm supp}\,\widehat{u_j(t)} \subset [0,\infty)$ for j=1,...,5, then we have for some $\varepsilon>0$,
	\begin{align*}
		\left\|\int_0^t e^{i(t-\tau)\Delta}  u_1u_2 u_3 u_4^*u_5^* (\tau)d\tau\right\|_{\mathcal{X}^s_{\sigma,\Delta}}\lesssim (T^{\varepsilon/2} +T) \prod^5_{j=1} \|u_j\|_{\mathcal{X}^s_{\sigma,\Delta}}.
	\end{align*}
	\end{prop}
\begin{proof}
	By duality, we only need to show
	\begin{align*}
		\left|\int_{\mathbb{R}^2}  u_1u_2  u_3 u_4^* u_5^* \bar{v} dxdt\right|\lesssim (T^{\varepsilon/2} +T)  \prod^5_{j=1} \|u_j\|_{\mathcal{X}^s_{\sigma,\Delta}} \|v\|_{\mathcal{Y}^{-s}_{-\sigma,\Delta}}.
	\end{align*}
By density argument, we can assume that the support sets of
	$\mathscr{F}_x u_j(t,\xi),\, \mathscr{F}_x v(t,\xi)$ contained in $[0,\infty)$ are compact. We have
\begin{align*}
		\mathcal{H} &: =  \int_{\mathbb{R}^2} u_1u_2  u_3 u_4^* u_5^* \bar{v} dxdt\\
		& = \int_{\mathbb{R}^2}  u_1u_2  u_3 u_4^* u_5^*    \overline{P_{< 1 } v} dxdt + \int_{\mathbb{R}^2} u_1u_2  u_3 u_4^* u_5^*    \overline{P_{\geq 1} v} dxdt ;=  \mathcal{H}_{lo} + \mathcal{H}_{hi}.
	\end{align*}
It is easy to see that
\begin{align}
		 |\mathcal{H}_{lo}|
		 & \leq  T \| P_{< 1 }  (u_1u_2  u_3 u_4^* u_5^* )  \|_{L^\infty_t L^2_x}  \|\overline{P_{< 1 } v} \|_{L^\infty_t L^2_x} \nonumber\\
 & \lesssim  T \prod^{5}_{i=1}\| P_{< 1 }  u_i \|_{L^\infty_t L^2_x}  \| P_{< 1 } v\|_{L^\infty_t L^2_x} \nonumber \\
  & \lesssim  T \prod^{5}_{i=1}\|u_i \|_{\mathcal{X}^s_{\sigma,\Delta}}  \|v\|_{\mathcal{Y}^{-s}_{-\sigma,\Delta}}. \label{dqunitiLow}
	\end{align}
 We can rewrite $\mathcal{H}_{hi}$ as
\begin{align*}
\int_{[0,\infty)^5\times\mathbb{R}^5}    \prod^3_{j=1}\widehat{u}_j(\xi_j,\tau_j) \prod^5_{j=4} \overline{\widehat{ u}}_j (\xi_j, - \tau_j) \overline{\widehat{P_{\geq 1} v}} (\xi_1+...+\xi_5 , \tau_1+...+\tau_5)  \prod^5_{j=1} d\xi_j   d\tau_j.
	\end{align*}
In order to simplify the notation, $ [0,\infty)^5\times\mathbb{R}^5$ and $d\xi_1...d\xi_5 d\tau_1...d\tau_5$ will be omitted in the expression of $\mathcal{J}_{hi}$. Noticing that ${\rm supp} \mathscr{F}_x{u}_j, \mathscr{F}_x{v} \subset [0,\infty)$ we can further rewrite $\mathcal{J}_{hi}$ as
	\begin{align*}
\mathcal{H}_{hi}
		  = \sum_{N\in {  2^{\mathbb{N}}}}\int   \prod^3_{j=1}2^{s|\xi_j|}\widehat{P_{< N }u}_j(\xi_j,\tau_j)  \prod^5_{j=4} 2^{s|\xi_j|} \overline{\widehat{P_{< N } u}}_j (\xi_j, - \tau_j) 2^{-s|\xi|}\overline{\widehat{P_N v}} (\xi , \tau ),
	\end{align*}
where we write $\xi_1+...+\xi_5= \xi$ and $ \tau_1+...+\tau_3 = \tau$ in $\mathcal{J}_{hi}$. Using $\xi = \xi_1+...+\xi_5$, one easily sees that  at least one frequency of  ${P_{< N }u}_1,..., {P_{< N }u}_5$  is greater than $2^{-5}N$. We can assume, without loss of generality that ${P_{< N }u}_1 = {P_{2^{-5}N \leq \cdot < N }u}_1$. It follows that
	\begin{align} \label{quintichigh}
|\mathcal{H}_{hi}| \leq  \sum_{N\in {  2^{\mathbb{N}}}} \prod^5_{j=2} \|2^{s|\nabla|} P_{<N} u_j\|_{L^6_{x,t}} \|P_{2^{-5}N \leq \cdot < N } 2^{s|\nabla|} u_1\|_{L^6_{x,t}}  \|2^{-s|\nabla|}  {P_N v}\|_{L^6_{x,t}}.
	\end{align}
Using the Bernstein, H\"older and Strichartz inequalities, one has that for $\varepsilon/3-\sigma <0$,
	\begin{align}
  \| P_{<N} u \|_{L^6_{x,t}}  & \lesssim  \sum_{0\leq k \leq N, \ k \ {\rm dyadic}} \langle k\rangle^{\varepsilon/3} \| P_{k} u\|_{L^6_t L^{6/(1+2\varepsilon)}_x}   \nonumber\\
 & \lesssim  T^{\varepsilon/6}  \sum_{0\leq k \leq N, \ k \ {\rm dyadic}}   \langle k \rangle^{\varepsilon/3} \|  P_{k} u\|_{L^{6/(1-\varepsilon)}_t L^{6/(1+2\varepsilon)}_x} \nonumber \\
 & \lesssim   T^{\varepsilon/6}  \left(\sum_{k\in 2^{\mathbb{N}}\cup \{0\}}  \langle k\rangle^{2\sigma}\| P_k u\|^2_{U^2 _\Delta} \right)^{1/2}.
 \label{quintilinear1est2}
	\end{align}
Taking $u= 2^{s|\nabla|}  u_j$ in \eqref{quintilinear1est2}, we have
	\begin{align}
  \| P_{<N}2^{s|\nabla|}  u_j \|_{L^6_{x,t}} \lesssim    T^{\varepsilon/6} \|u_j\|_{\mathcal{X}^s_{\sigma,\Delta}}.
 \label{quintilinear1est2a}
	\end{align}
Inserting the estimate \eqref{quintilinear1est2a} into \eqref{quintichigh}, then applying the Cauchy-Schwarz, Strichartz inequalities and Proposition \ref{basicproperty},  we obtain that
	\begin{align} \label{quintichigh2aa}
|\mathcal{H}_{hi}| & \lesssim  T^{\varepsilon/2} \prod^5_{j=2} \|u_j\|_{\mathcal{X}^s_{\sigma,\Delta}} \sum_{N\in {  2^{\mathbb{N}}}}   \|P_{2^{-5}N \leq \cdot < N } 2^{s|\nabla|} u_1\|_{L^6_{x,t}}  \|2^{-s|\nabla|}  {P_N v}\|_{L^6_{x,t}}  \nonumber\\
& \lesssim  T^{\varepsilon/2} \prod^5_{j=2} \|u_j\|_{\mathcal{X}^s_{\sigma,\Delta}} \left(\sum_{N\in {  2^{\mathbb{N}}}\cup\{0\} }  \langle N\rangle^{2\sigma} \|P_{2^{-5}N \leq \cdot < N } 2^{s|\nabla|} u_1\|^2_{U^2_\Delta} \right)^{1/2} \nonumber\\
 &  \quad \quad \quad \times \left(\sum_{N\in {  2^{\mathbb{N}}} \cup\{0\}}  \langle N\rangle^{-2\sigma}\|2^{-s|\nabla|}  {P_N v}\|^2_{V^2_\Delta} \right)^{1/2}  \nonumber\\
 & \lesssim  T^{\varepsilon/2} \prod^5_{j=1} \|u_j\|_{\mathcal{X}^s_{\sigma,\Delta}}   \|v\|_{\mathcal{Y}^{-s}_{- \sigma,\Delta}}.
	\end{align}
Combining the estimates  \eqref{dqunitiLow} and  \eqref{quintichigh2aa}, we have the result, as desired.
\end{proof}

	\begin{prop}[Trilinear estimate]\label{dnonlinearpart}
	Let $T>0$, $s\leq 0$, $\sigma>0$. Assume that $u_j\in X^s_{\sigma,\Delta}$, $u_j(t) = 0 $ for $t<0$ or $t\geq T$, and ${\rm supp}\,\widehat{u_j(t)} \subset [0,\infty)$ for j=1,2,3, then we have for some $\varepsilon>0$,
	\begin{align*}
		\left\|\int_0^t e^{i(t-\tau)\Delta}  u_1u_2 \partial_x u_3^*  (\tau)d\tau\right\|_{\mathcal{X}^s_{\sigma,\Delta}}\lesssim (T^{\varepsilon/2} +T) \prod^3_{j=1} \|u_j\|_{\mathcal{X}^s_{\sigma,\Delta}}.
	\end{align*}
	\end{prop}
	\begin{proof}[\textbf{Proof}]
	By duality, we only need to show
	\begin{align*}
		\left|\int_{\mathbb{R}^2}  u_1u_2 \partial_x u_3^*  \bar{v} dxdt\right|\lesssim (T^{\varepsilon/2} +T) \prod^3_{j=1} \|u_j\|_{\mathcal{X}^s_{\sigma,\Delta}}  \|v\|_{\mathcal{Y}^{-s}_{-\sigma,\Delta}}.
	\end{align*}
By density argument, we can assume that the support sets of
	$\mathscr{F}_x u_j(t,\xi),\, \mathscr{F}_x v(t,\xi)$ contained in $[0,\infty)$ are compact. 	
\begin{align*}
		\mathcal{J} &: =  \int_{\mathbb{R}^2}  u_1u_2\partial_x u^*_3  \bar{v} dxdt\\
		& = \int_{\mathbb{R}^2}  u_1u_2\partial_x u^*_3  \overline{P_{< 1 } v} dxdt + \int_{\mathbb{R}^2}  u_1u_2 \partial_x u^*_3  \overline{P_{\geq 1} v} dxdt ;=  \mathcal{J}_{lo} + \mathcal{J}_{hi}.
	\end{align*}
It is easy to see that
\begin{align}
		 |\mathcal{J}_{lo}|
		 & \leq  T \| P_{< 1 }  (u_1u_2\partial_x u^*_3)  \|_{L^\infty_t L^2_x}  \|\overline{P_{< 1 } v} \|_{L^\infty_t L^2_x} \nonumber\\
 & \lesssim  T \prod^{3}_{i=1}\| P_{< 1 }  u_i \|_{L^\infty_t L^2_x}  \| P_{< 1 } v\|_{L^\infty_t L^2_x} \nonumber \\
  & \lesssim  T \prod^{3}_{i=1}\|u_i \|_{\mathcal{X}^s_{\sigma,\Delta}}  \|v\|_{\mathcal{Y}^{-s}_{-\sigma,\Delta}}. \label{dLow}
	\end{align}
Using the same way as in Proposition \ref{nonlinearpart},  we can rewrite $\mathcal{J}_{hi}$ as
\begin{align*}
\int_{[0,\infty)^3\times\mathbb{R}^3} i\xi_3  \, \widehat{u}_1(\xi_1,\tau_1) \widehat{u}_2(\xi_2,\tau_2) \overline{\widehat{u}}_3 (\xi_3, - \tau_3) \overline{\widehat{P_{\geq 1} v}} (\xi_1+\xi_2+\xi_3 , \tau_1+\tau_2+\tau_3)   d\xi_1d\xi_2d\xi_3 d\tau_1d\tau_2d\tau_3.
	\end{align*}
If there is no confusion, $ [0,\infty)^3\times\mathbb{R}^3$ and $d\xi_1d\xi_2d\xi_3 d\tau_1d\tau_2d\tau_3$ will be omitted in the expression of $\mathcal{J}_{hi}$. Noticing that ${\rm supp} \mathscr{F}_x{u}_j, \mathscr{F}_x{v} \subset [0,\infty)$ we can rewrite $\mathcal{J}_{hi}$ as
	\begin{align*}
\mathcal{J}_{hi}
		  = \sum_{N\in {  2^{\mathbb{N}}}}\int   i \xi_3  \, \widehat{P_{< N }u}_1(\xi_1,\tau_1) \widehat{P_{< N } u}_2(\xi_2,\tau_2) \overline{\widehat{P_{< N } u}}_3 (\xi_3, - \tau_3) \overline{\widehat{P_N v}} (\xi , \tau ),
	\end{align*}
where we write $\xi_1+\xi_2+\xi_3= \xi$ and $ \tau_1+\tau_2+\tau_3 = \tau$ in $\mathcal{J}_{hi}$. Recall that
\begin{align} \label{dhighestmodulation}
  \sum_{i=1,2} |\xi^2_i +\tau_i|+ |\xi^2_3-\tau_3| + |\xi^2+\tau| \geq  2 (\xi_1+\xi_3)(\xi_2+\xi_3).
\end{align}
According to the lower bound of the highest modulation and noticing that $\xi=\xi_1+\xi_2+\xi_3$, one can decompose
	\begin{align*}
		\mathcal{J}_{hi}
		  = & \sum_{N\in {  2^{\mathbb{N}}}}\int  i \xi_3  \, \widehat{2^{s |\nabla|} P_{< N }u}_1(\xi_1,\tau_1) \widehat{2^{s |\nabla|} P_{< N } u}_2(\xi_2,\tau_2) \overline{2^{s |\nabla|} \widehat{P_{2^{-3}N\leq \cdot < N } u}}_3 (\xi_3, -\tau_3) \overline{\widehat{2^{-s |\nabla|} P_N v}} (\xi, \tau)  \\
& + \sum_{N\in {  2^{\mathbb{N}}}}\int  i \xi_3  \, \widehat{2^{s |\nabla|} P_{< N }u}_1(\xi_1,\tau_1) \widehat{2^{s |\nabla|} P_{< N } u}_2(\xi_2,\tau_2) \overline{2^{s |\nabla|} \widehat{P_{\leq  2^{-3}N} u}}_3 (\xi_3, -\tau_3) \overline{2^{-s |\nabla|} \widehat{P_N v}} (\xi, \tau)\\
: = & \mathcal{J}_1 +\mathcal{J}_2.
	\end{align*}

{\it Step 1.} We  estimate $\mathcal{J}_1$. By \eqref{dhighestmodulation}, the highest modulation of  ${P_{< N }u}_1$, ${P_{< N } u}_2$, $P_{2^{-3}N\leq \cdot < N } u_3$,  ${P_N v} $ satisfies
\begin{align} \label{dhighestmodulation1}
   \sum_{i=1,2} |\xi^2_i +\tau_i|+ |\xi^2_3-\tau_3| + |\xi^2+\tau| \gtrsim N^2 .
\end{align}
First, we consider the case that ${P_{< N }u}_1$ has the highest modulation, i.e., $|\xi^2_1 +\tau_1|\geq N^2$.
By H\"older's inequality, we have
\begin{align}
|\mathcal{J}_1|
		  & \leq    \sum_{N\in { 2^{\mathbb{N}}}}  N  \|2^{s |\nabla|} P_{< N }u_1\|_{L^2_{x,t}}  \|2^{s |\nabla|} P_{< N } u_2\|_{L^4_tL^\infty_x}  \nonumber\\
&  \ \ \ \ \ \ \ \ \ \  \times \| 2^{s|\nabla|}   P_{2^{-3}N\leq \cdot < N } u_3 \|_{L^4_tL^\infty_x}  \|2^{-s |\nabla|} P_N v \|_{L^\infty_tL^2_x}. \label{dL1est1}
\end{align}
Notice that $\sigma >0$.  Using Bernstein's inequality, Proposition \ref{strichartz} and $U^2_\Delta \subset U^{4+}_\Delta $, for any $u $ with $u(t)=0$ for $t>T$, $0<\varepsilon <\sigma$,  one has that
\begin{align}
 \| P_{< N } u \|_{L^4_tL^\infty_x}
 & \lesssim  \sum_{0\leq k \leq N, \ k \ {\rm dyadic}}  \| P_{k} u\|_{L^4_t L^{\infty}_x}   \nonumber\\
 & \lesssim  T^{\varepsilon/2}  \sum_{0\leq k \leq N, \ k \ {\rm dyadic}}   \langle k \rangle^{\varepsilon-\sigma} \| \langle\nabla\rangle^{\sigma} P_{k} u\|_{L^{4/(1-2\varepsilon)}_t L^{1/\varepsilon}_x} \nonumber \\
 & \lesssim   T^{\varepsilon/2}  \left(\sum_{k\in 2^{\mathbb{N}}\cup \{0\}}  \langle k\rangle^{2\sigma}\| P_k u\|^2_{U^{4/(1-2\varepsilon)} _\Delta} \right)^{1/2}\nonumber\\
 & \lesssim   T^{\varepsilon/2}  \left(\sum_{k\in 2^{\mathbb{N}}\cup \{0\}}  \langle k\rangle^{2\sigma}\| P_k u\|^2_{U^{2} _\Delta} \right)^{1/2}.
 \label{dL1est2}
\end{align}
Taking $u=2^{s |\nabla|}  u_2 $ in \eqref{dL1est2}, we have
\begin{align}
 \| P_{< N } u \|_{L^4_tL^\infty_x}
 & \lesssim   T^{\varepsilon/2}  \|u_2\|_{\mathcal{X}^s_{\sigma,\Delta}}.
 \label{dL1est2a}
\end{align}
In view of the dispersion modulation decay estimate \eqref{dispersiondecay} and \eqref{dhighestmodulation1},
\begin{align}
  \|2^{s |\nabla|} P_{< N }u_1\|_{L^2_{x,t}}   \leq  N^{-1} \|2^{s |\nabla|}   u_1\|_{V^2_{\Delta}} \leq N^{-1} \|u_1\|_{\mathcal{X}^s_{\sigma,\Delta}}.  \label{dL1est3}
\end{align}
Inserting the estimates of \eqref{dL1est2a} and \eqref{dL1est3} into \eqref{dL1est1}, we have
\begin{align}
|\mathcal{J}_1|
		  & \lesssim T^{\varepsilon/2}  \|u_1\|_{\mathcal{X}^s_{\sigma,\Delta}}  \|u_2\|_{\mathcal{X}^s_{\sigma,\Delta}}  \nonumber\\
&  \ \ \ \ \ \ \    \times \sum_{N\in {\{0\}\cup 2^{\mathbb{N}}}}      \| 2^{s|\nabla|}    P_{2^{-3}N\leq \cdot < N }  u_3 \|_{U^2_\Delta}  \|2^{-s |\nabla|}   P_N v \|_{L^\infty_tL^2_x}.
 \label{dL1est4}
\end{align}
Using the embedding $  V^2_{-,rc} \subset L^\infty_tL^2_x$ and Cauchy-Schwarz inequality, it follows from \eqref{dL1est4} that
 \begin{align}
|\mathcal{J}_1|
  & \leq T^\varepsilon    \prod^3_{i=1}\|  u_i\|_{\mathcal{X}^s_{\sigma, \Delta, T}}     \| v\|_{\mathcal{Y}^{-s}_{-\sigma, \Delta, T} } \label{dL1est5}
\end{align}
If $ P_{ < N } u_2$ or $ P_{2^{-3}N\leq \cdot < N } u_3$ has the highest modulation, we can use the same way as above  to obtain \eqref{dL1est5}. When $P_N v$ has the highest modulation. one need to use the highest modulation of  $P_N v$, $2^{-s|\nabla|} P_N v$ should be equipped with $L^2_{x,t}$ norm. We have
\begin{align}
|\mathcal{J}_1|
		  & \leq    \sum_{N\in { 2^{\mathbb{N}}}}  N  \|2^{s |\nabla|} P_{< N }u_1\|_{ L^4_tL^\infty_x }  \|2^{s |\nabla|} P_{< N } u_2\|_{L^4_tL^\infty_x}  \nonumber\\
&  \ \ \ \ \ \ \ \ \ \  \times \| 2^{s|\nabla|}   P_{2^{-3}N\leq \cdot < N } u_3 \|_{L^\infty_tL^2_x}  \|2^{-s |\nabla|} P_N v \|_{L^2_tL^2_x}, \label{dL1est1aa}
\end{align}
we can use similar way as above to get  \eqref{dL1est5}.

{\it Step 2.} We consider the estimate of $\mathcal{J}_2$. Let us decompose $\mathcal{J}_2$ by
	\begin{align*}
\mathcal{J}_{2} = &  \sum_{N\in {  2^{\mathbb{N}}}}\int i\xi_3 \widehat{P_{< N }u}_1(\xi_1,\tau_1) \widehat{P_{\leq 2^{-3}N} u}_2(\xi_2,\tau_2) \overline{\widehat{P_{\leq  2^{-3}N} u}}_3 (\xi_3, -\tau_3) \overline{\widehat{P_N v}} (\xi, \tau)\\
&+  \sum_{N\in {  2^{\mathbb{N}}}}\int i\xi_3 \widehat{P_{< N }u}_1(\xi_1,\tau_1) \widehat{P_{2^{-3} N\leq \cdot< N } u}_2(\xi_2,\tau_2) \overline{\widehat{P_{\leq  2^{-3}N} u}}_3 (\xi_3, -\tau_3) \overline{\widehat{P_N v}} (\xi, \tau)\\
: = & \mathcal{J}_{21} +\mathcal{J}_{22}.
	\end{align*}
In $\mathcal{J}_{21}$,  $\xi =\xi_1+\xi_2+ \xi_3$ implies that $\xi_1\geq 2^{-2}N$, so one can rewrite $\mathcal{J}_{21}$ as
\begin{align*}
\mathcal{J}_{21} = &  \sum_{N\in {  2^{\mathbb{N}}}}\int i\xi_3 \widehat{P_{2^{-2}N\leq \cdot < N }u}_1(\xi_1,\tau_1) \widehat{P_{\leq 2^{-3}N} u}_2(\xi_2,\tau_2) \overline{\widehat{P_{\leq  2^{-3}N} u}}_3 (\xi_3, -\tau_3) \overline{\widehat{P_N v}} (\xi, \tau)\\
= &  \sum_{N\in {  2^{\mathbb{N}}}}\int i\xi_3 2^{s\xi_1}\widehat{P_{2^{-2}N\leq \cdot < N }u}_1(\xi_1,\tau_1)  2^{s\xi_2}\widehat{P_{\leq 2^{-3}N} u}_2(\xi_2,\tau_2) 2^{s\xi_3} \overline{\widehat{P_{\leq  2^{-3}N} u}}_3 (\xi_3, -\tau_3)  2^{-s\xi }\overline{\widehat{P_N v}} (\xi, \tau).
	\end{align*}
Applying H\"older's inequality and  Lemma \ref{Bilinear}, we have
\begin{align*}
|\mathcal{J}_{21}|  &  \leq   \sum_{N\in    2^{\mathbb{N} }}  \| 2^{s|\nabla|} P_{2^{-2}N\leq \cdot < N }u_1   2^{s|\nabla|}  P_{\leq 2^{-3}N} u_2 \|_{L^2_{x,t}} \| 2^{s|\nabla|}   P_{\leq  2^{-3}N} \partial_x u_3 (-\cdot,\cdot)   2^{-s|\nabla|}  P_N v \|_{L^2_{x,t}} \\
 &  \leq  T^{\varepsilon/2} \sum_{N\in    2^{\mathbb{N} }} N^{-1+2\varepsilon}  \| 2^{s|\nabla|} P_{2^{-2}N\leq \cdot < N }u_1 \|_{V^2_\Delta}  \|2^{s|\nabla|}  P_{\leq 2^{-3}N} u_2 \|_{V^2_\Delta} \\
  & \ \ \ \  \times \| 2^{s|\nabla|}   P_{\leq  2^{-3}N} \partial_x u_3 (-\cdot,\cdot)\|_{V^2_\Delta}  \| 2^{-s|\nabla|}  P_N v \|_{V^2_\Delta}.
 \end{align*}
 Using the definition of $V^2$ and the Cauchy-Schwarz inequality, we have
 \begin{align}
 \| 2^{s|\nabla|}   P_{\leq  2^{-3}N} \partial_x u_3 (-\cdot,\cdot)\|_{V^2_\Delta}
 & \lesssim  \sum_{0\leq k \leq 2^{-3}N, \ k \ {\rm dyadic}} \langle k \rangle  \|2^{s|\nabla|} P_{k} u_3 \|_{V^2_\Delta}   \nonumber\\
 & \lesssim  \langle N \rangle^{1-\sigma}  \left(\sum_{k\in 2^{\mathbb{N}}\cup \{0\}}  \langle k\rangle^{2\sigma}\| P_k u\|^2_{U^{2}_\Delta} \right)^{1/2}   \nonumber\\
 & \lesssim  \langle N \rangle^{1-\sigma}  \|u_3\|_{\mathcal{X}^s_{\sigma, \Delta, T}} ,
 \label{ddL1est2} \\
 \|2^{s|\nabla|}  P_{\leq 2^{-3}N} u_2 \|_{V^2_\Delta}  & \lesssim     \|u_2 \|_{\mathcal{X}^s_{\sigma, \Delta, T}}.
 \label{dddL1est2}
\end{align}
 For $2\varepsilon-\sigma  <0$,
 \begin{align*}
|\mathcal{J}_{21}|    &  \lesssim  T^{\varepsilon/2}  \|u_2 \|_{\mathcal{X}^s_{\sigma, \Delta, T}}  \|u_3 \|_{\mathcal{X}^s_{\sigma, \Delta, T}} \sum_{N\in    2^{\mathbb{N} }} N^{ 2\varepsilon- \sigma }  \| 2^{s|\nabla|}\langle \nabla\rangle^{\sigma}  u_1 \|_{V^2_\Delta}
     \| 2^{-s|\nabla|} \langle \nabla\rangle^{-\sigma}   v \|_{V^2_\Delta} \\
  &  \lesssim  T^{\varepsilon/2}
   \prod^3_{i=1}\|  u_i\|_{\mathcal{X}^s_{\sigma, \Delta, T}}     \| v\|_{\mathcal{Y}^{-s}_{-\sigma, \Delta, T} }.
	\end{align*}

For the estimate of $\mathcal{J}_{22}$, one can rewrite it as
\begin{align*}
\mathcal{J}_{22} = & \sum_{N\in {  2^{\mathbb{N}}}}\int i\xi_3 \widehat{P_{2^{-4}N \leq \cdot < N }u}_1(\xi_1,\tau_1) \widehat{P_{2^{-3} N\leq \cdot< N } u}_2(\xi_2,\tau_2) \overline{\widehat{P_{\leq  2^{-3}N} u}}_3 (\xi_3, -\tau_3) \overline{\widehat{P_N v}} (\xi, \tau) \\
& + \sum_{N\in {  2^{\mathbb{N}}}}\int i\xi_3 \widehat{P_{ \leq 2^{-4}N }u}_1(\xi_1,\tau_1) \widehat{P_{2^{-3} N\leq \cdot< N } u}_2(\xi_2,\tau_2) \overline{\widehat{P_{\leq  2^{-3}N} u}}_3 (\xi_3, -\tau_3) \overline{\widehat{P_N v}} (\xi, \tau)\\
:=& \mathcal{J}_{221}  +  \mathcal{J}_{222}.
\end{align*}
In $\mathcal{J}_{221}$, the highest modulations of $  P_{2^{-4}N \leq \cdot < N }u_1$, $ P_{2^{-3} N\leq \cdot< N } u_2$,
$ P_{\leq  2^{-3}N} u_3$ and $ P_N v$ satisfy
$$
\sum_{i=1,2} |\xi^2_i +\tau_i| + |\xi^2_3 -\tau_3| + |\xi^2+\tau| \geq \xi_1\xi_2  \gtrsim N^2 .
$$
So, we can repeat the procedure as in Step 1 to obtain that
$$
|\mathcal{J}_{221}| \lesssim  T^{\varepsilon}
   \prod^3_{i=1}\|  u_i\|_{X^s_{\sigma, \Delta, T}}     \| v\|_{Y^{-s}_{-\sigma, \Delta, T} }.
$$
Noticing that $\mathcal{J}_{222}$ has similar structure to $\mathcal{J}_{21}$ and we can use Lemma \ref{Bilinear} to obtain that
$$
|\mathcal{J}_{222}| \lesssim  T^{\varepsilon/2}
   \prod^3_{i=1}\|  u_i\|_{X^s_{\sigma, \Delta, T}}     \| v\|_{Y^{-s}_{-\sigma, \Delta, T} }.
$$
So, we obtain that
\begin{align} \label{dL2est}
|\mathcal{J}_{22}| \leq |\mathcal{J}_{221}| + |\mathcal{J}_{222}| \lesssim  T^{\varepsilon/2}
   \prod^3_{i=1}\|  u_i\|_{X^s_{\sigma, \Delta, T}}     \| v\|_{Y^{-s}_{-\sigma, \Delta, T} }.
\end{align}
Summarizing the estimates of $\mathcal{J}_{21}$ and $\mathcal{J}_{22}$, we have
\begin{align} \label{dL2est1}
|\mathcal{J}_{2}|    \lesssim  T^{\varepsilon/2}
   \prod^3_{i=1}\|  u_i\|_{X^s_{\sigma, \Delta, T}}     \| v\|_{Y^{-s}_{-\sigma, \Delta, T} }.
\end{align}
By \eqref{dLow}, \eqref{dL1est5}, \eqref{dL2est1} and duality, we have the result for the case $T< 1 $, as desired. If $T>1$, it suffices to use that $T^{1/4} \sim (T^{\varepsilon/4} +T^{1/4})$  in Lemma \ref{Bilinear}.
	\end{proof}

\subsection{Proof of Theorem \ref{mainresult2}}

	\begin{proof}[\textit{Proof of Theorem \ref{mainresult2}}]
		First, we show the existence and uniqueness of the local solutions for initial data in $E^s_\sigma$.  Let us consider the mapping
		\begin{align*}
			\mathscr{T}: u\mapsto \chi_{[0,\infty)}(t)e^{it\Delta} u_0 -i \int_0^t e^{i(t-\tau)\Delta}\chi_{[0,T)} \left(\alpha  u^2 \partial_x u^* + \frac{1}{2} u^3 (u^*)^2 \right)(\tau)d\tau.
		\end{align*}
By \eqref{Inclusion}, Propositions \ref{linearpart}, \ref{dquintilinear} and \ref{dnonlinearpart}, we have for $0<T<1$,
$$
\|\mathscr{T}u\|_{\mathcal{X}^s_{\sigma, \Delta}} \lesssim \|u_0\|_{E^s_\sigma }+T^{\varepsilon/2} (\|u\|^3_{\mathcal{X}^s_{\sigma, \Delta}}+ \|u\|^5_{\mathcal{X}^s_{\sigma, \Delta}})
$$
and
$$
\|\mathscr{T}u-\mathscr{T} v\|_{\mathcal{X}^s_{\sigma, \Delta}}\lesssim T^{\varepsilon/2} (\|u\|_{\mathcal{X}^s_{\sigma, \Delta}}^2+\|v\|_{\mathcal{X}^s_{\sigma, \Delta}}^2 +\|u\|^4_{\mathcal{X}^s_{\sigma, \Delta}} +\|v\|^4_{\mathcal{X}^s_{\sigma, \Delta}} )\|u-v\|_{\mathcal{X}^s_{\sigma, \Delta}}.
$$
Applying the contraction mapping principle, we can obtain that
		\begin{align} \label{dNNLSINT}
			  u(t) = \chi_{[0,\infty)}(t)e^{it\Delta} u_0 -i\int_0^t e^{i(t-\tau)\Delta} \chi_{[0,T)}\left(\alpha  u^2 \partial_x u^* + \frac{1}{2} u^3 (u^*)^2 \right)(\tau)d\tau.
		\end{align}
has a unique solution  $u\in \mathcal{X}^s_{\sigma, \Delta}$ if
\begin{align}
T = \frac{1}{(2C(\|u_0\|^2_{E^s_\sigma}+ \|u_0\|^4_{E^s_\sigma}))^{2/\varepsilon}},  \label{dTlocal}
\end{align}
and the solution satisfies
\begin{align}
\|u\|_{L^\infty(0, T ; E^s_\sigma) \cap \mathcal{X}^s_{\sigma, \Delta}} \leq  2C \|u_0\|_{E^s_\sigma}.   \label{dTlocal2}
\end{align}
If $\|u_0\|_{E^s_\sigma}\ll 1$, we can replace $T^{\varepsilon/2}$ by $T$ in the above arguments and let
\begin{align}
T =  \frac{1}{ 4C\|u_0\|^2_{E^s_\sigma} } >1.  \label{dTlocalsmall}
\end{align}
By Propositions \ref{dquintilinear} and \ref{dnonlinearpart}, we have
$$
\|\mathscr{T}u\|_{\mathcal{X}^s_{\sigma, \Delta}} \lesssim \|u_0\|_{E^s_\sigma }+T ( \|u\|^3_{\mathcal{X}^s_{\sigma, \Delta}} +\|u\|^5_{\mathcal{X}^s_{\sigma, \Delta}}).
$$
Applying the same way as above we can get the local solution $u\in \mathcal{X}^s_{\sigma, \Delta}$.

Next, we consider the global existence and uniqueness of the solutions. Recall that $\mathrm{supp}~\widehat{u}_0\subset [\varepsilon_0,\infty)$ for some $\varepsilon_0 >0$. Denote $u_{0,\lambda}(x) = \lambda^{1/2} u_0(\lambda x)$.
		Following Lemma \ref{scaling},  we have
\begin{align*}
			 \|u_{0,\lambda}\|_{E^{s}_\sigma }  \leq \lambda^{\sigma} 2^{s \lambda \varepsilon_0/2} \|u_0 \|_{E^{s}_\sigma } \leq 2^{s \lambda \varepsilon_0/4} , \ \ \lambda \geq \lambda_0:= \lambda_0 (\sigma, \varepsilon_0, \|u_0\|_{E^{s}_\sigma }).
\end{align*}
By the local well-posedness, we see that NdNLS \eqref{dNNLSINT} with initial data $u_{0,\lambda}$  has a unique solution
$$
u_\lambda \in L^\infty(0, T_\lambda; E^s_\sigma ) \cap  \mathcal{X}^s_{\sigma, \Delta} ,  \ \ T_\lambda =   1/4C \|u_{0,\lambda} \|^2_{E^{s}_\sigma } ,
$$
$$
\|u_\lambda\|_{L^\infty(0, T_\lambda; E^s_\sigma) \cap  \mathcal{X}^s_{\sigma, \Delta} }  \leq  2 C  \|u_{0,\lambda} \|_{E^{s}_\sigma } .
$$
By choosing $\lambda \geq \lambda_1:= \lambda_1(\sigma, \varepsilon_0,  \|u_0\|_{E^{s}_\sigma })$, we see that
$$
T_\lambda \geq 2^{\sqrt{\lambda}}, \ \ \  \|u_\lambda\|_{L^\infty(0, T_\lambda; E^s_\sigma) \cap \mathcal{X}^s_{\sigma, \Delta} }  \leq  2^{- \sqrt{\lambda}}.
$$
Then, one sees that
\begin{align} \label{scalinguu}
 u(t, x) = \lambda^{-1/2} u_\lambda(\lambda^{-2}  t,\lambda^{-1}  x)
\end{align}
is the solution of \eqref{dNNLSINT} with initial data $u_{0}$. By Lemma \ref{scaling}, for any $t\leq  \lambda^2  2^{ \sqrt{\lambda}  } $ and $\lambda\geq \lambda_1$,
$$
\| u(t)\|_{E^{s\lambda}_\sigma } = \lambda^{-1/2} \|u_\lambda(\lambda^{-2}  t,\lambda^{-1} \cdot )\|_{E^{s\lambda}_\sigma }  \leq  C  2^{ - \sqrt{\lambda } } .
 $$
Further, we can show that $u\in X^{s\lambda}_{\sigma, \Delta}$ if $\lambda\geq \lambda_1$.

\begin{lemma} Let $s\leq 0, \, \sigma>0$. Let $ u_\lambda (t,x) =u(\lambda^2t, \lambda x)$. We have
$$
\|u\|_{\mathcal{X}^{s\lambda}_{\sigma,\Delta}} \leq  \lambda^{1/2  } (\ln \lambda) \|u_\lambda\|_{\mathcal{X}^{s}_{\sigma,\Delta}}, \ \ \lambda \gg 1.
$$

\end{lemma}

\begin{proof}
 In view of $B_{\mathfrak{t}}(f_{\lambda^{-1}}, \, v) = \lambda  B_{\lambda^{-2}\mathfrak{t}}(f, \, v_\lambda) $, we see that
$$
B (f_{\lambda^{-1}}, \, v) = \lambda  B (f, \, v_\lambda).
$$
It follows that
$$
|B (f_{\lambda^{-1}}, \, v) | \leq  \lambda  \|f\|_{\mathcal{X}^{s}_{\sigma,\Delta}} \| v_\lambda\|_{\mathcal{Y}^{-s}_{-\sigma,\Delta}}
$$
Using the definition of the norm on $\mathcal{Y}^{-s}_{-\sigma,\Delta}$,  we have
 \begin{align*}
 \|v_\lambda\|^2_{\mathcal{Y}^{-s}_{-\sigma,\Delta}}: = \left( \sum_{N\leq \lambda}  +  \sum_{N >\lambda}\right) \langle N\rangle^{-2\sigma} \| 2^{-s|\nabla|} P_N  v_\lambda \|^2_{V^2_{\Delta}} := I+II .
\end{align*}
In view of the definition of $V^2_\Delta$,
 \begin{align*}
\|   2^{-s|\nabla|} P_N  v_\lambda \|^2_{V^2_{\Delta}}
& =  \lambda^{-2} \| \chi_{[N/2,N]}(\xi)  2^{- s|\xi|}  e^{-\mathrm{i} t|\xi|^2} \widehat{v(\lambda^2 t)} (\lambda^{-1} \xi )\|^2_{V^2} \\
&= \lambda^{-1} \| 2^{-s\lambda |\nabla|} P_{N/\lambda}  v \|^2_{V^2_{\Delta}}.
\end{align*}
It follows that
\begin{align*}
II & \leq    \lambda^{-1}\sum_{N >\lambda}  \langle N/\lambda \rangle^{-2\sigma} \| 2^{-s\lambda|\nabla|} P_{N/\lambda}  v \|^2_{V^2_{\Delta}}  \lesssim   \lambda^{-1}\| v \|^2_{\mathcal{Y}^{-s\lambda}_{-\sigma,\Delta}},\\
I & \leq    \lambda^{-1}\sum_{N \leq \lambda}    \| 2^{-s\lambda|\nabla|} P_{N/\lambda}  v \|^2_{V^2_{\Delta}}  \lesssim \lambda^{-1} (\ln \lambda) \| 2^{-s\lambda|\nabla|} P_{0}  v \|^2_{V^2_{\Delta}}.
\end{align*}
Hence, we have
\begin{align*}
\| v_\lambda\|_{\mathcal{Y}^{-s}_{\sigma,\Delta}} \lesssim \lambda^{-1/2} (\ln \lambda )   \| v \|_{\mathcal{Y}^{-s\lambda}_{-\sigma,\Delta}}.
\end{align*}
By duality, we get the result.
\end{proof}

Now, taking $\lambda =j \in \mathbb{N}$ and $j\geq \lambda_1 $, we obtain a unique solution of NdNLS \eqref{dNNLSINT}
$
u\in L^\infty(0,   2^{\sqrt{ j}} ; E^{sj}_\sigma ) \cap \mathcal{X}^{sj}_{\sigma, \Delta, 2^{\sqrt{j}} }
$
satisfying
$$
\| u \|_{L^\infty(0,   2^{ \sqrt{j} } ; E^{sj}_\sigma) \cap \mathcal{X}^{sj}_{\sigma, \Delta , 2^{\sqrt{j}} }}   \leq     2^{-\sqrt{j}}.
 $$
The result follows.
 \end{proof}

\subsection{Boundness of gauge transform in $\mathcal{X}^{s}_{\sigma, \Delta }$ and proof of Corollary \ref{mainresult3}}

 \begin{prop} \label{xBDgauge}
Let $s\leq 0$, $\sigma >0$. Suppose that $u\in \mathcal{X}^s_{\sigma,\Delta}$ satisfying  $u(t) = 0 $ for $t<0$ or $t\geq T$,  ${\rm supp} \widehat{u} (\xi) \subset \{\xi: \ \xi\geq \varepsilon_0\}$ for some $\varepsilon_0 >0$. Then we have $\mathcal{G}(u) \in \mathcal{X}^s_{\sigma,\Delta}$ and there exists $C:= C(\varepsilon_0,\sigma)>0$ such that
\begin{align}
\|\mathcal{G}(u)\|_{\mathcal{X}^s_{\sigma,\Delta}} \leq \, T \exp (C \|u\|^2_{\mathcal{X}^s_{\sigma,\Delta}}) \|u\|_{\mathcal{X}^s_{\sigma, \Delta}}. \label{xbdgauge1}
\end{align}
\end{prop}
\begin{proof}
By duality, it suffices to show that
\begin{align}\label{xbdgaugedual}
\left |\int_{\mathbb{R}^2} \mathcal{G}(u) \overline{v} dxdt \right|  \leq  T  \exp (C \|u\|^2_{\mathcal{X}^s_{\sigma,\Delta}}) \|u\|_{\mathcal{X}^s_{\sigma, \Delta, T}} \|v\|_{\mathcal{Y}^{-s}_{-\sigma, \Delta, T}}.
\end{align}
By Taylor's expansion,
\begin{align} \label{xtaylor}
\int_{\mathbb{R}^2} \mathcal{G}(u) \overline{v} dxdt  = \sum^\infty_{k=0} \frac{(-\delta)^k}{k!}  \int_{\mathbb{R}^2} u(\partial^{-1}_x(uu^*))^k \overline{v} dxdt .
\end{align}
We have
\begin{align}
 \int_{\mathbb{R}^2} u(\partial^{-1}_x(uu^*))^k \overline{v} dxdt  &= \sum_{N\in 2^{\mathbb{N}} \cup \{0\}}  \int_{\mathbb{R}^2} u(\partial^{-1}_x(uu^*))^k \overline{P_N v} dxdt \sum_{N\in 2^{\mathbb{N}} \cup \{0\}} A_N. \label{xbdgauge2}
\end{align}
For simply, we  write $2^{s|\nabla|} u =u_1 , \ 2^{s|\nabla|} u =u_2$. Using  \eqref{bdgauge3} and   ${\rm supp} \ \widehat{u} \subset (0,\infty)$, one has that
\begin{align}
A_N =   \int_{\mathbb{R}^2} P_{<N}u_1 (\partial^{-1}_x(P_{<N} u_1 P_{<N} u_2))^k \overline{P_N 2^{-s|\nabla|} v} dxdt
\end{align}
The estimate of $A_0$ is easy, by H\"older's
 Young's and Strichartz' inequalities,
\begin{align} \label{A0}
|A_0 | & \leq T \|P_{<1}u_1\|_{L^\infty_tL^2}\|P_{<1}2^{-s|\nabla|} v\|_{L^\infty_tL^2} \|\partial^{-1}_x(P_{<1}u_1 P_{<1}u_2) \|^k_{L^\infty_{x,t}}  \nonumber\\
& \leq T C^k_{\varepsilon_0,\sigma} \|u\|^{2k+1}_{\mathcal{X}^s_{\sigma,\Delta}} \|v\|^{2k}_{\mathcal{Y}^{-s}_{\sigma,\Delta}} .
\end{align}
Next, we estimate $A_N$ for $N\geq 2$. Noticing that
$$
\int_{\mathbb{R}^2} P_{<N/(2k+1)}u_1 (\partial^{-1}_x(P_{<N/(2k+1)} u_1 P_{<N/(2k+1)} u_2))^k \overline{P_N 2^{-s|\nabla|} v} dxdt =0,
$$
one only needs to estimate
\begin{align}
& A_{N1}:= \int_{\mathbb{R}^2} P_{ N/(2k+1)\leq \cdot <N}u_1 (\partial^{-1}_x(P_{<N} u_1 P_{<N} u_2))^k \overline{P_N 2^{-s|\nabla|} v} dxdt ,  \nonumber\\
& A_{N2}:= \int_{\mathbb{R}^2} P_{<N}u_1\partial^{-1}_x(P_{ N/(2k+1)\leq \cdot <N} u_1 P_{<N} u_2)  (\partial^{-1}_x(P_{<N} u_1 P_{<N} u_2))^{k-1} \overline{P_N 2^{-s|\nabla|} v} dxdt,   \nonumber\\
& A_{N3}:= \int_{\mathbb{R}^2} P_{<N}u_1\partial^{-1}_x(P_{<N} u_1 P_{ N/(2k+1)\leq \cdot <N} u_2)  (\partial^{-1}_x(P_{<N} u_1 P_{<N} u_2))^{k-1} \overline{P_N 2^{-s|\nabla|} v} dxdt.   \nonumber
\end{align}
In view of H\"older's inequality, and following \eqref{xbdgauge3}, we have
\begin{align}
 |A_{N1}|
& \leq \, T \| P_{ N/(2k+1)\leq \cdot <N}u_1 \|_{L^\infty_tL^2}\|P_{N}2^{-s|\nabla|} v\|_{L^\infty_tL^2} \|\partial^{-1}_x(P_{<N}u_1 P_{<N}u_2) \|^k_{L^\infty_{x,t}}  \nonumber\\
& \leq  T C^{k}_{\varepsilon_0,\sigma} \|u\|^{2k}_{\mathcal{X}^s_{\sigma,\Delta}} \| P_{ N/(2k+1)\leq \cdot <N}u_1 \|_{L^\infty_tL^2}\|P_{N}2^{-s|\nabla|} v\|_{L^\infty_tL^2}
. \label{xBesovest3}
\end{align}
By  \eqref{xBesovest3} and H\"older's inequality, we have
\begin{align}
\sum_{N\in 2^{\mathbb{N}}} |A_{N1}|
 & \lesssim  T C^{k}_{\varepsilon_0,\sigma} \|u\|^{2k}_{\mathcal{X}^s_{\sigma,\Delta}}  \|v\|_{\mathcal{Y}^{-s}_{-\sigma, \Delta}} \left(\sum_{N\in 2^{\mathbb{N}}} N^{2\sigma} \| P_{ N/(2k+1)\leq \cdot <N}u_1 \|^2_{L^\infty_tL^2}  \right)^{1/2}.  \label{xBesovest4}
\end{align}
By Young's inequality, one can estimate
\begin{align}
&  \left(\sum_{N\in 2^{\mathbb{N}}} N^{2\sigma} \| P_{ N/(2k+1)\leq \cdot <N}u_1 \|^2_{L^\infty_tL^2}  \right)^{1/2} \nonumber\\
 & \leq  \left(\sum_{N\in 2^{\mathbb{N}}} \left( \sum_{1\leq l \leq [\log^{2k+1}_2]+1 } N^{ \sigma} \| P_{2^{-l+1} N } u_1 \|_{L^\infty_tL^2} \right)^2 \right)^{1/2} \nonumber\\
 & \leq  \sum_{1\leq l \leq [\log^{2k+1}_2]+1 }    2^{l\sigma} \|u\|_{\mathcal{X}^s_{\sigma,\Delta}}\nonumber\\
 & \lesssim (2k+1)^{\sigma}   \|u\|_{\mathcal{X}^s_{\sigma,\Delta}}.
  \label{xBesovest5}
\end{align}
Summarizing the estimates of \eqref{xBesovest4} and \eqref{xBesovest5}, we have
\begin{align} \label{xBesovest6}
\sum_{N\in 2^{\mathbb{N}}} |A_{N1}| \lesssim   T C^{k}_{\varepsilon_0,\sigma} (2k+1)^{\sigma}   \|u\|^{2k+1}_{\mathcal{X}^s_{\sigma,\Delta}}  \|v\|_{\mathcal{Y}^{-s}_{-\sigma, \Delta}} .
\end{align}
Now let us estimate $A_{N2}$.
By H\"older's inequality,   we have
\begin{align}
 |A_{N2}|
\leq & \ T \| P_{ <N}u_1 \|_{L^\infty_tL^2}\|P_{N}2^{-s|\nabla|} v\|_{L^\infty_tL^2}  \|\partial^{-1}_x(P_{<N}u_1 P_{<N}u_2) \|^{k-1}_{L^\infty_{x,t}}  \nonumber\\
& \times  \|\partial^{-1}_x(P_{ N/(2k+1)\leq \cdot <N} u_1 \,  P_{<N}u_2) \|_{L^\infty_{x,t}}  \nonumber\\
\leq  & \  T C^{k-1}_{\varepsilon_0,\sigma} \|u\|^{2k-1}_{\mathcal{X}^s_{\sigma,\Delta}} \|\partial^{-1}_x(P_{ N/(2k+1)\leq \cdot <N} u_1 \,  P_{<N}u_2) \|_{L^\infty_{x,t}} \|P_{N}2^{-s|\nabla|} v\|_{L^\infty_tL^2}
. \label{xBesovest7}
\end{align}
Using Bernstein's inequality, we have
 \begin{align}
  \|\partial^{-1}_x & (P_{ N/(2k+1)\leq \cdot <N} u_1 \,  P_{<N}u_2) \|_{L^\infty_{x,t}} \nonumber\\
     &  \leq (2k+1)   \| P_{ N/(2k+1)\leq \cdot <N} u_1 \,  P_{<N}u_2  \|_{L^\infty_t L^1_x} \nonumber\\
   &  \leq (2k+1)   \| P_{ N/(2k+1)\leq \cdot <N} u_1 \|_{L^\infty_t L^2_x}   \| P_{<N}u_2  \|_{L^\infty_t L^2_x} . \label{xBesovest8}
\end{align}
Inserting the estimate \eqref{xBesovest8} into  \eqref{xBesovest7}, we have
\begin{align}
 |A_{N2}|
& \leq  T (2k+1) C^{k}_{\varepsilon_0,\sigma} \|u\|^{2k}_{\mathcal{X}^s_{\sigma,\Delta}} \| P_{ N/(2k+1)\leq \cdot <N}u_1 \|_{L^\infty_tL^2}\|P_{N}2^{-s|\nabla|} v\|_{L^\infty_tL^2}.
\label{xBesovest9}
\end{align}
Repeating the procedure as in the estimate of $\sum_{N\in 2^{\mathbb{N}}} |A_{N1}|$, we can get that
\begin{align} \label{xBesovest10}
\sum_{N\in 2^{\mathbb{N}}} |A_{N2}| \lesssim   T C^{k}_{\varepsilon_0,\sigma} (2k+1)^{\sigma+1}   \|u\|^{2k+1}_{\mathcal{X}^s_{\sigma,\Delta}}  \|v\|_{\mathcal{Y}^{-s}_{-\sigma, \Delta}} .
\end{align}
The estimate of $\sum_{N\in 2^{\mathbb{N}}} |A_{N3}|$ is the same as in \eqref{xBesovest10} and the details are omitted.  Summarizing the estimates of \eqref{A0}, \eqref{xBesovest6} and \eqref{xBesovest10},
\begin{align}
 \left|\int_{\mathbb{R}^2} u(\partial^{-1}_x(uu^*))^k \overline{v} dxdt \right| \lesssim   T C^{k}_{\varepsilon_0,\sigma} (2k+1)^{\sigma+1}   \|u\|^{2k+1}_{\mathcal{X}^s_{\sigma,\Delta}}  \|v\|_{\mathcal{Y}^{-s}_{-\sigma, \Delta}}  . \label{xbdgauge200}
\end{align}
Noticing that $(2k+1)^{1+\sigma} \leq C^k$, by \eqref{xtaylor} and  \eqref{xbdgauge200} we have
\begin{align} \label{xtaylor100}
\left|\int_{\mathbb{R}^2} \mathcal{G}(u) \overline{v} dxdt \right| \leq  T  \sum^\infty_{k=0} \frac{|C\delta|^k}{k!}   \|u\|^{2k+1}_{\mathcal{X}^s_{\sigma,\Delta}}  \|v\|_{\mathcal{Y}^{-s}_{-\sigma, \Delta}}.
\end{align}
Hence, we have \eqref{xbdgaugedual}, as desired.
\end{proof}

{\it Proof of Corollary \ref{mainresult3}.}	 By Theorem \ref{mainresult2}, we see that Eq. \eqref{aNdNLSequiv} with initial data $v_0=   u_0 \exp (  \frac{\alpha}{4} \partial^{-1}_x  (u_0  {u}_0^* )  )$ has a unique solution $v$ satisfying the results of Theorem \ref{mainresult2}. Further, in view of Propositions \ref{Equivndnls}, \ref{BDgauge} and  \ref{xBDgauge}, we have the result, as desired. $\hfill\Box$

\subsection{Ill-posedness}

Similar to the NNLS, we can show that NdNLS is ill-posed in any $E^{s'}_{\sigma'}$ if $u_0 \in E^s_{\sigma}$ ($s'\leq s<0, \, \sigma'\leq \sigma, \, \sigma >0$) does not satisfy the condition ${\rm supp} \, \widehat{u}_0 \subset [0,\infty)$, where the ill-posedness means that the solution map from $E^s_{\sigma}$ into $E^{s'}_{\sigma'}$ is not $C^3$. In fact, one can take the same initial data as for the NNLS:
 $$
 \widehat{\varphi}(\xi) = 2^{-sk/2} (\chi_{k+I}(\xi) + \chi_{-2k+2I}(\xi) ),  \ \ I =[1/8,\, 1/4], \ 2I = [1/4, 1/2], \ \ k\gg 1.
 $$
 Noticing that
 If \eqref{NNLS} has a solution $u_\varepsilon$ and the solution map $\varepsilon \varphi \to u_\varepsilon$ is $C^3$, then we have
\begin{align}
& \frac{\partial u_\varepsilon(t)}{\partial \varepsilon}|_{\varepsilon=0} =  e^{it\partial^2_x} \varphi, \ \  \frac{\partial^2 u_\varepsilon(t)}{\partial \varepsilon^2}|_{\varepsilon=0} =0, \label{dillposed1} \\
& \frac{\partial^3 u_\varepsilon(t)}{\partial\varepsilon^3 }|_{\varepsilon=0} = \mathrm{i} 6\alpha \int^t_0 e^{\mathrm{i}(t-\tau)\partial^2_x} (e^{\mathrm{i}\tau \partial^2_x} \varphi)  (e^{\mathrm{i}\tau \partial^2_x} \partial_x \varphi) (e^{\mathrm{i}\tau \partial^2_x} \varphi)^* d\tau.  \label{dillposed2}
\end{align}
Let us observe that the only difference between \eqref{illposed2} and \eqref{dillposed2} is that there is an additional $\partial_x$ in the expression of $\frac{\partial^3 u_\varepsilon(t)}{\partial\varepsilon^3 }|_{\varepsilon=0}$ in \eqref{dillposed2}, which means that
\begin{align}
& \left|\frac{\widehat{\partial^3 u_\varepsilon(t)} (\xi)}{\partial\varepsilon^3 }|_{\varepsilon=0}  \right| \nonumber  \\
& \ \ \ \gtrsim  \left| \chi_{[1/2,1]}(\xi)\int_{\mathbb{R}^2}  \frac{\xi_2(e^{\mathrm{i}2 t (\xi-\xi_1)(\xi-\xi_2)} -1) }{2(\xi-\xi_1)(\xi-\xi_2) }  \varphi(\xi_1)  \varphi(\xi_2) \varphi(\xi-\xi_1-\xi_2)   d\xi_1d\xi_2 \right|:= R_d(\xi).\label{dillposed3}
\end{align}
Namely, comparing \eqref{illposed3} with \eqref{dillposed3}, we see that there is an additional $\xi_2$ in \eqref{dillposed3}.  Since $\xi_2 \sim k$ in \eqref{dillposed3}, it will increase the lower bound of $R_d(\xi)$. So, using the same way as in Section \ref{Illposed}, we can get that the solution map is not $C^3$.

\section{General NdNLS} \label{sectgNdNLS}

The general nonlocal  derivative NLS \eqref{gNdNLS} covers NdNLS \eqref{NdNLS} and the following (cf. \cite{Zh2018})
\begin{equation*}
		 \mathrm{i} u_t + \partial^2_x u + \alpha\, \partial_x (u u^*u)   =0,  \  u(x,0) = u_0(x).
\end{equation*}
Using the nonlocal gauge transform \eqref{gauge1}, we can show that gNdNLS is equivalent to the following equation
\begin{equation}\label{gNdNLSequiv}
		 \mathrm{i} v_t + \partial^2_x v - (\alpha-\beta) \,  v^2 \partial_x v^* -\frac{\alpha^2}{2}\left(\alpha- \frac{3\beta}{2}\right) v^3 (v^*)^2=0,  \  v(x,0) = v_0(x).
\end{equation}
\begin{prop} \label{gEquivndnls}
If $u$ is a smooth solution of gNdNLS \eqref{gNdNLS}, then $v=\mathcal{G}(u)$ with $\alpha =-2\delta$ is the smooth solution of \eqref{gNdNLSequiv} with initial data $v_0= \mathcal{G}(u_0)$. Conversely, if $v$ is a smooth solution of \eqref{gNdNLSequiv}, then $u= v \mathcal{E}^{-1}(v)$ with initial data $u_0 = v_0 \mathcal{E}^{-1}(v_0)$ is a smooth solution of \eqref{gNdNLS}.
\end{prop}
Noticing that \eqref{gNdNLSequiv} can be solved by using the same way as in Section \ref{GloNDNLS}, we can get the same results for Eqs. \eqref{gNdNLSequiv} and  \eqref{gNdNLS} as in Theorem \ref{mainresult2} and Corollary \ref{mainresult3}. \\

\textbf{Acknowledgments.} The authors are supported in part by the NSFC, grant 12171007.
	
	\phantomsection 
	\bibliographystyle{amsplain}

\vspace{10pt}

	\scriptsize\textsc{Jie Chen: School of Sciences, Jimei University,  Xiamen, 361021 and Institute of Applied Physics and Computational Mathematics, PO Box 8009, P.~R.~of China.}
	
	\textit{E-mail address}: \verb"jiechern@pku.edu.cn"
	\vspace{10pt}
	
\scriptsize\textsc{Yufeng Lu and Baoxiang Wang: School of Sciences, Jimei University,  Xiamen, 361021 and
 School of  Mathematical Sciences, Peking University, Beijing 100871 P.~R.~of China.}
	
	\textit{E-mail address}: \verb"wbx@math.pku.edu.cn"	


\end{document}